\documentclass[11pt,a4paper,oneside]{amsart}

\usepackage{amsmath}
\usepackage{amsfonts}
\usepackage{amssymb}
\usepackage{amsthm}
\usepackage{mathtools}
\usepackage{bbm}
\usepackage{booktabs}
\usepackage{threeparttable}
\usepackage[utf8x]{inputenc}
\usepackage[english]{babel}
\usepackage{enumerate}
\usepackage{graphicx} 
\usepackage{tikz}
\usetikzlibrary{shapes,arrows}
\usetikzlibrary{fit,backgrounds,shapes.geometric}
\usetikzlibrary{positioning}
\usetikzlibrary{calc}
\usetikzlibrary{decorations}
\usetikzlibrary{decorations.markings}
\usetikzlibrary{patterns}

\usepackage{url}
\usepackage{enumitem}

\usepackage{libertine}

\usepackage[T1]{fontenc}

\usepackage[boxed,vlined,norelsize,oldcommands]{algorithm2e}
\SetKwInOut{Input}{input}
\SetKwInOut{Output}{output}
\SetArgSty{rm}
\SetFuncSty{tt}
\setalcapskip{1ex}

\newcommand{\Z}{{\mathbb Z}}
\newcommand{\RR}{{\mathbb R}}
\newcommand{\ZZ}{{\mathbb Z}}
\newcommand{\zz}{{\ZZ}}
\newcommand{\zzp}{\zz_{\geq 0}}
\newcommand{\rr}{{\RR}}

\newcommand{\qq}{{\mathbb Q}}

\newcommand{\eps}{\varepsilon}

\newcommand{\vones}{\mathbbm{1}}

\DeclareMathOperator{\PGL}{PGL}
\DeclareMathOperator{\PSL}{PSL}
\DeclareMathOperator{\PGGL}{P\Gamma L}
\DeclareMathOperator{\AGGL}{A\Gamma L}
\DeclareMathOperator{\AGL}{AGL}
\DeclareMathOperator{\ASL}{ASL}
\DeclareMathOperator{\LL}{L}

\DeclareMathOperator{\conv}{conv}

\DeclareMathOperator{\lin}{span}

\DeclareMathOperator{\stab}{Stab}

\DeclareMathOperator{\fix}{Fix}
\DeclareMathOperator{\vertices}{vert}

\newcommand\core{\operatorname{core}}

\newcommand\GL{{\operatorname{GL}}} 

\newcommand{\Symmet}[1]{\mathcal{S}_{#1}}
\newcommand{\Cycl}[1]{\mathcal{C}_{#1}}

\theoremstyle{plain}
\newtheorem{theorem}{Theorem}
\newtheorem{proposition}[theorem]{Proposition}
\newtheorem{corollary}[theorem]{Corollary}
\newtheorem{lemma}[theorem]{Lemma}
\newtheorem{conjecture}[theorem]{Conjecture}

\theoremstyle{definition}
\newtheorem{definition}[theorem]{Definition}

\newtheorem{example}[theorem]{Example}

\newtheorem{remark}[theorem]{Remark}

\newcommand{\zzk}[1]{\zz_{(#1)}^n}
\newcommand{\hk}[1]{H_{\vones,#1}}
\newcommand{\proj}[2]{\ensuremath{{#1}|_{#2}}}
\newcommand{\ac}[2]{\ensuremath{{#2}{#1}}}

\newcommand{\smallSetOf}[2]{\{#1\,|\,#2\}}
\newcommand{\bigSetOf}[2]{\left\{#1\,:\,#2\right\}}

\newcommand{\card}[1]{\ensuremath{\left\lvert {#1}\right\rvert}}

\newcommand{\inv}[1]{\ensuremath{#1^{-1}}}
\newcommand{\norm}[1]{\ensuremath{\left\|{#1}\right\|}}
\newcommand{\scp}[2]{\ensuremath{\left \langle {#1}, {#2}\right \rangle}}
\newcommand{\cin}[2]{#1^{\left(#2\right)}}
\newcommand{\floor}[1]{\ensuremath{\left\lfloor{#1}\right\rfloor}}

\newcommand{\tp}{\top}
\newcommand{\defeq}{\vcentcolon=}

\newcommand{\defterm}{\emph}

\newcommand{\setOneTo}[1]{\ensuremath{\left[{#1}\right]}}

\let\Gamma\varGamma

\DeclareMathOperator{\bw}{bw}

\title{On Lattice-Free Orbit Polytopes}
\author[K. Herr, T. Rehn, and A. Sch\"{u}rmann]{Katrin Herr, Thomas Rehn, and Achill Sch\"{u}rmann}


\address{%
Department of Mathematics,
Technische Universität Darmstadt,  
Dolivostraße 15,
64293 Darmstadt, 
Germany}
\email{herr@mathematik.tu-darmstadt.de}

\address{%
initOS GmbH \& Co.\ KG,
Hegelstraße 28,
39104 Magdeburg,
Germany}
\email{thomas.rehn@research.initos.com}

\address{%
Institute for Mathematics,
University of Rostock,
18051 Rostock,
Germany}
\email{achill.schuermann@uni-rostock.de}

\thanks{Research by the first author was partially supported by Studienstiftung des deutschen
  Volkes and the paper was partially completed while the third author was visiting the Institute for Mathematical Sciences, National University of Singapore, in 2013.}

\begin{document}

\begin{abstract}
  Given a permutation group acting on coordinates of $\rr^n$, we consider 
  lattice-free polytopes that are the convex hull of an orbit of one integral vector. 
  The vertices of such polytopes are called \emph{core points} and they play a key role 
  in a recent approach to exploit symmetry in integer convex optimization problems. 
  Here, naturally the question arises, for which groups the number of core points 
  is finite up to translations by vectors fixed by the group. 
  In this paper we consider transitive permutation groups and prove this type of finiteness 
  for the $2$-homogeneous ones.
  We provide tools for practical computations of core points 
  and obtain a complete list of representatives for all $2$-homogeneous groups up to degree twelve.
  For transitive groups that are not $2$-homogeneous we conjecture that 
  there exist infinitely many core points up to translations by the all-ones-vector.
  We prove our conjecture for two large classes of groups: 
  For imprimitive groups and groups that have an irrational invariant subspace.
\end{abstract}

\maketitle
\setcounter{tocdepth}{1}
\tableofcontents

\section{Introduction}

Let $\Gamma\leq\Symmet{n}$ be a permutation group acting on $\rr^n$ by permuting coordinates. 
We consider orbit polytopes that are convex hulls $\conv(\Gamma z)$ of an orbit of an integral vector $z\in\zz^n$.
Such an orbit polytope is called lattice-free, when its vertices
are the only integral vectors in the polytope.
We note that lattice-free polytopes (as used in \cite{MR1300509,MR1737330})
are sometimes called empty lattice polytopes (see \cite{MR1709397,MR1737331}). 
We call the integral vertices of lattice-free orbit polytopes 
\emph{core points} with respect to~$\Gamma$ (cf. \cite{HRS2012}). 
These core points play an important role in symmetric integer convex optimization 
as a $\Gamma$-symmetric convex set contains an integral point if and only if it
contains a core point of~$\Gamma$ (cf. \cite[Theorem~4]{HRS2012}).

Core points can therefore 
be used to design algorithms that take advantage of available symmetries.
This is in particular the case when the number of core points is finite 
up to translations by vectors in the fixed space of~$\Gamma$.
In this case it is even possible to use a na\"ive approach based 
on enumeration of core points,
beating state-of-the-art optimization software for selected problems, 
as shown in~\cite{HRS2012}.
For such an approach, however, a full list of core points in needed.
In this paper we therefore not only address the fundamental question 
for which groups~$\Gamma$ a finiteness result holds,
but we also provide computational 
techniques to obtain full lists of core points in such cases.
As far as possible, we apply our tools to groups~$\Gamma$ of small degree
and provide full lists of core points,
which could potentially be used for future computations.
It should be noted though that such lists are not at all necessary
for the design of core point based, symmetry exploiting algorithms. 
Even groups~$\Gamma$ which do not have finitely many core points
(up to translations by vectors in the fixed space) may allow
the use of good approximations or parametrizations of them.

We focus on transitive permutation groups~$\Gamma$ only, 
i.e., groups such that all coordinates lie in the same orbit.  
This is a first necessary step in a study of more general groups as every
permutation group relates to a product of transitive groups.  
Next to the design of new core point based algorithms for
integer convex optimization problems, a detailed study of core
points of intransitive groups 
are major open tasks for future research 
(cf.~\cite{KatrinPhD, ThomasPhD}).

\bigskip

Our paper is organized as follows.
In Section~\ref{sec:basic-definitions} we introduce some notation and recall elementary
properties of core points. Using the John ellipsoid~\cite{MR0030135} we prove in
Section~\ref{sec:coreset-close-to-invariant-subspace} that core points of a given group
are always close to an invariant subspace of the group.  
It is a well known fact from representation theory that the space $\rr^n$ can be
decomposed into a direct sum of pairwise orthogonal invariant subspaces of the given
group.  A transitive group always fixes the one-dimensional linear subspace spanned by the
all-ones vector~$\vones$ and therefore also preserves its $(n-1)$-dimensional orthogonal
complement~$\vones^{\perp}$.  Therefore, every transitive permutation group has at least
these two invariant subspaces.  In the following sections we distinguish two fundamentally
different cases.

In Section~\ref{sec:finite-coresets} we study groups for which the $(n-1)$-dimensional
invariant subspace~$\vones^{\perp}$ cannot be decomposed into smaller invariant subspaces,
that is, we consider groups acting \emph{irreducibly} on~$\vones^{\perp}$. By
Cameron~\cite[Lemma~2]{MR0294471}, these are precisely the $2$-homogeneous groups.  We
show that in this case there exist only finitely many core points up to translation by the
all-ones vector.  This allows in principle to obtain a complete list of core points (up to
translation).  We provide mathematical tools for an exhaustive computer search, which we
perform up to dimension twelve (see Table~\ref{tab:coresets-cand-elimination}).

For the other case, that is, for groups having more than two invariant subspaces, we
conjecture that there are infinitely many core points up to translation (see
Conjecture~\ref{conj:finite-coreset-iff-group-irreducible}).  In
Section~\ref{sec:infinite-coresets} we prove our conjecture for two major cases:
imprimitive groups and groups which have an irrational invariant subspace.
Figure~\ref{fig:finite-vs-infinite-coresets} depicts an overview of the groups whose core
points we study in detail.  Despite convincing computational evidence (see
Section~\ref{sec:infinite-coresets}) for the remaining cases,
a complete proof for
Conjecture~\ref{conj:finite-coreset-iff-group-irreducible} is still missing.
\begin{figure}[htb]
  \newcommand{\boundellipse}[3]
{(#1) ellipse (#2 and #3)
}

\begin{tikzpicture}
\small{

\begin{scope}
\path[clip](0,0) arc (-90:90:4.2cm and 1.8cm);
\path[draw,fill=black!10] (0,0) to[out=90,in=180] (4.7,2.25) -- (4.7,0)--cycle;
\path[draw,fill=black!10] (0,3.6) to[out=-90,in=180] (4.7,1.35) -- (4.7,3.6)--cycle;
\path[draw] (0,0) to[out=90,in=180] (4.7,2.25) -- (4.7,0)--cycle;

\path[clip] (1,0) -- (4.7,0)--(4.7,3.6) -- (1,3.6) --cycle;
\path[clip] (0,0) to[out=90,in=180] (4.7,2.25) -- (4.7,0)--cycle;
\path[draw,pattern=vertical lines,pattern color=black] (0,3.6) to[out=-90,in=180] (4.7,1.35) -- (4.7,3.6)--cycle;
\end{scope}


\filldraw[thick,color= black!25, draw=black] (0,0) node (a){}  arc (270:90:4.2cm and 1.8cm)
node(b){} -- (0cm,1.8cm) node[rectangle, left, inner sep = 0.2cm,text
width=2.8cm,color=black] {groups
  acting\\irreducibly on $\vones^{\perp}$}-- (0,0);
\draw[thick](a) arc (-90:90:4.2cm and 1.8cm) -- (0cm,0.8cm) node[rectangle, right, inner sep = 0.7cm,text
width=2.5cm,color=black] {groups with an irrational inv. subspace} -- (0cm,1.8cm) node[rectangle, right, inner sep = 0.3cm,text
width=2.8cm,color=black] (m) {\large{\textbf{?}}} -- (0cm,2.8cm) node[rectangle, right, inner sep = 0.7cm,text
width=2.5cm,color=black] {imprimitive groups} -- (a);

}
\end{tikzpicture}
 \caption{Finite vs. infinite core sets}
 \label{fig:finite-vs-infinite-coresets}
\end{figure}
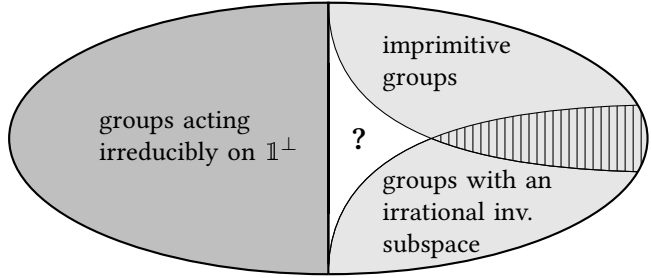

\newpage

\section{Basic definitions and core points}\label{sec:basic-definitions}

\subsection{Permutation groups and representations}\label{sec:permutation-group-basics}

We denote by $\scp{\cdot}{\cdot}$ the standard inner product in $\rr^n$.
The orthogonal projection of a vector $x$ onto a linear subspace $V$ is denoted by $\proj{x}{V}$.

By $\Symmet{n}$ we denote the symmetric group on the set $\setOneTo{n}:=\{1,\dots,n\}$.
Let $\Gamma \leq \Symmet{n}$ be a permutation group.
Since it acts on the set $\setOneTo{n}$, we define its \emph{degree} as~$n$.
We say that $\Gamma$ is \emph{transitive} if for every $x,y \in \setOneTo{n}$ there is a permutation $\gamma \in \Gamma$ with $\gamma x = y$.
In other words, all elements of $\setOneTo{n}$ lie in the same $\Gamma$-orbit.
More generally, we say that $\Gamma$ is \emph{$k$-transitive} for a $k \in \setOneTo{n}$ if for every two $k$-tuples $(x_1,\dots,x_k), (y_1,\dots,y_k) \in \setOneTo{n}^k$, with $x_i \not= x_j$ and $y_i \not= y_j$ for $i\not=j$, 
there is a permutation $\gamma \in \Gamma$ with $\gamma x_i = y_i$ for all $i \in \setOneTo{k}$.
The group $\Gamma$ is called \emph{$k$-homogeneous} if for every two subsets $X, Y \subset \setOneTo{n}$ with $k$ elements there is a permutation $\gamma \in \Gamma$ with $\gamma X = Y$.
Thus, every $k$-transitive group also is $k$-homogeneous.

If there exists a non-trivial partition $\setOneTo{n} = \bigsqcup_{i=1}^m \Omega_i$ with $2 \leq m \leq n-1$ and a permutation $\sigma$ of $[m]$ such that $\Gamma \Omega_i = \Omega_{\sigma(i)}$ for all $i$, we call the $\Omega_i$ \emph{blocks of imprimitivity}.
If no such partition exists, we say that $\Gamma$ is \emph{primitive}.

We do not distinguish between a permutation group and its canonical linear representation which acts on $\rr^n$ by permuting coordinates.
We call a subspace $V \subset \rr^n$ an \emph{invariant subspace} for $\Gamma$ if it is setwise fixed, i.e. $\Gamma V = V$.
Note the difference to the \emph{fixed space} $\fix(\Gamma)$, which is the subspace of all pointwise fixed elements of $\rr^n$.
We can always decompose $\rr^n$ into a direct sum of orthogonal $\Gamma$-invariant subspaces because the linear representation of $\Gamma$ is an orthogonal group.
For $\Gamma \leq \Symmet{n}$ the space spanned by the all-ones vector $\vones$ is always an invariant subspace, which is even fixed pointwise.

\subsection{Core points}\label{sec:core-points}

Core points were first studied in~\cite{bhj-2011} with respect to full symmetric and
alternating groups. The following definition is taken from~\cite{HRS2012}, which
generalizes the definition of core points from~\cite{bhj-2011} to arbitrary subgroups of~$\Symmet{n}$.

\begin{definition} \label{def:corepoint} Given a group $\Gamma\leq\Symmet{n}$, a
  \emph{core point with respect to $\Gamma$} is an integral point $z\in\ZZ^n$ such that
  the convex hull of its $\Gamma$-orbit does not contain any further integral points, that
  is, $\conv(\Gamma z)\cap \ZZ^n=\Gamma z$.
  Phrased differently, the \emph{orbit polytope} of $z$ with respect to~$\Gamma$, that is, $\conv(\Gamma z)$, is lattice-free.
\end{definition}

\begin{remark}\label{rem:corepoints-up-to-translation}
Translation by~$\vones$ commutes with any element of~$\Gamma$, therefore the polytope $\conv(\Gamma (z+\vones))$ is the translate $\vones + \conv(\Gamma z)$.
Thus, it suffices to study core points up to translation by~$\vones$.
\end{remark}

There are two canonical ways to choose representatives.
The first is studying core points in the affine hyperplanes $\hk{0},\dots,\hk{n-1}$ where $\hk{k}\coloneqq\smallSetOf{x\in\RR^n}{\scp{x}{\vones}=k}$.
We refer to the set of integer points in these hyperplanes as \emph{layers}~$\zzk{k} \defeq \zz^n \cap \hk{k}$ with \emph{index}~$k$.
The second way is to select zero-based representatives according to the following definition.
\begin{definition}\label{def:zero-based}
 A point $z\in\zz^n$ is called \emph{zero-based} if all its coordinates are non-negative and at least one coordinate is zero.
\end{definition}

At the end of this section we recall the only previously known result about core points of transitive groups.
For the full symmetric and alternating group on~$n$ variables the following characterization of core points was proven in~\cite{bhj-2011}.
Since any subgroup $\Gamma'\leq\Gamma$ inherits the core points of~$\Gamma$, the core points with respect to~$\Gamma=\Symmet{n}$ are core points with respect to any group $\Gamma'\leq\Symmet{n}$.
For this reason we call them \emph{universal core points}.

\begin{example}[Universal core points]\label{ex:universal-core-points}
  Let $\Gamma$ be the full symmetric or the alternating group on~$n$ variables. For each $k\in\setOneTo{n}$, the core points in the affine hyperplane~$\hk{k}$ are precisely the vertices of the hypersimplex: 
\[\core_\Gamma(\hk{k})=\bigSetOf{\sum_{i\in T}e_i}{T\text{ a $k$-element subset of
  }[n]}\enspace.\] 
  Thus, each core point with respect to~$\Gamma$ is an integral point
with coordinates in $\{t,t+1\}$ for some $t\in\ZZ$.
\end{example}

\section{Core points are close to invariant subspaces}\label{sec:coreset-close-to-invariant-subspace}

In this section we will show that core points are always close to an invariant subspace of the group.
To prove this we use a well-known theorem from convex geometry (\cite{MR0030135}, see also \cite{1014.52001}).
\begin{theorem}[John ellipsoid \cite{MR0030135}]\label{thm:john-ellipsoid}
 Let $K \subset \rr^n$ be a convex body, i.e., $K$ is compact and convex with non-empty interior.
 Among all ellipsoids containing~$K$ there exists a unique ellipsoid~$E$ of minimal volume.
 Further, a scaled version of $E$ is in turn contained in $K$:
 $$t + \frac{1}{n}E \subseteq K \subseteq E,$$
 where $t \subseteq \rr^n$ is a suitable translation vector that depends on the center of~$E$.
 The scaling factor $\frac{1}{n}$ for $E$ is optimal as the case of a simplex shows.
\end{theorem}
This ellipsoid is called the \emph{minimal enclosing ellipsoid} of $K$.
For orbit polytopes this ellipsoid can be computed as follows.
Let $\Gamma\leq\Symmet{n}$ be a permutation group.
Recall from Section~\ref{sec:permutation-group-basics} that we can decompose $\rr^n$ into pairwise orthogonal $\Gamma$-invariant subspaces.
We have
$$\rr^n = \bigoplus_{i=1}^m V_i$$
where each $V_i\subseteq \rr^n$ is setwise preserved by $\Gamma$ 
(i.e., $\gamma v \in V_i$ for all $\gamma\in \Gamma$ and $v\in V_i$) 
and is irreducible 
(i.e., $V_i$ does not contain a proper invariant subspace).
Note that, depending on the group, this decomposition may not be unique, which is a well-known fact in representation theory (see, for instance, \cite{MR0450380}).
The minimal enclosing ellipsoid of an orbit polytope is closely related to invariant subspaces as \cite[Thm.~2.2]{MR2178312} shows.
In the following $\|\cdot\|$ always refers to the Euclidean norm.
\begin{theorem}[\cite{MR2178312}]\label{thm:orbit-polytope-ellipsoid}
 Let $\Gamma\leq\Symmet{n}$ be a transitive permutation group.
 Let $z \in \zzk{k}$ be such that the dimension of the orbit polytope of $z$ is maximal, i.e., $\dim \conv \Gamma z = n - 1$.
 Then there exists a decomposition $\rr^n = \lin\vones \oplus \bigoplus_{i=1}^m V_i$ of $\rr^n$ into the fixed space $\lin\vones$  
and other $\Gamma$-invariant invariant subspaces~$V_i$ such that
 the minimal enclosing ellipsoid of the orbit polytope $\conv(\Gamma z)$ is given by
 \begin{equation}\label{eq:minimal-enclosing-ellipsoid}
  \frac{k}{n}\vones + \left\{  x \in \hk{0} \;:\; \sum_{i=1}^m (\dim V_i) \frac{\| \proj{x}{V_i} \|^2}{\| \proj{z}{V_i} \|^2} \leq n-1 \right\}.
 \end{equation}
\end{theorem}

\begin{remark}\label{rem:minimal-enclosing-ellipsoid-dim}
  The ellipsoid given by~\eqref{eq:minimal-enclosing-ellipsoid} is contained in the affine hyperplane~$\hk{k}$ and thus is not full-dimensional, having dimension $n-1$.
\end{remark}

\begin{remark}
  If the decomposition of $\rr^n$ into $\Gamma$-invariant subspaces is unique, then the minimal enclosing ellipsoid is also uniquely determined by the formula in the theorem.
If there are multiple decompositions, only one of these leads to the minimal enclosing ellipsoid.
\end{remark}

\begin{theorem}\label{thm:core-points-close-to-invariant-subspace}
 Let $\Gamma\leq\Symmet{n}$ be a transitive permutation group. 
 Then there exist a constant $C(n)$ depending only on the dimension~$n$, 
 such that for every core point~$z$ with respect to~$\Gamma$
 there exists a $\Gamma$-invariant subspace $V$ of $\rr^n$ different from $\fix(\Gamma) = \lin\vones$ such that $\left\|\proj{z}{V}\right\| \leq C(n)$.
\end{theorem}
\begin{proof}
 We use the two preceding theorems in this section to find a necessary condition under which the orbit polytope $P \defeq\conv \Gamma z$ contains integral points. We first consider $z\in\zzk{k}$ with a fixed $k$.
 By Theorem~\ref{thm:orbit-polytope-ellipsoid} there is a decomposition $\rr^n = \fix(\Gamma) \oplus \bigoplus_{i=1}^m V_i$  of $\rr^n$ into $\Gamma$-invariant subspaces related to the minimal enclosing ellipsoid of $P$.
 If $\norm{\proj{z}{V_i}} = 0$ for one subspace $V_i$, then nothing remains to be shown.
 So we assume that all projections $\proj{z}{V_i}$ have positive norm.
 Then the dimension of the polytope~$P$ is $n-1$.
 By Theorem~\ref{thm:orbit-polytope-ellipsoid} we know that the minimal enclosing ellipsoid of the orbit polytope $P$ is 
 $$\frac{k}{n}\vones + \left\{  x \in \hk{0} \;:\; \sum_{i=1}^m (\dim V_i) \frac{\| \proj{x}{V_i} \|^2}{\| \proj{z}{V_i} \|^2} \leq n-1 \right\}.$$
 By John's Theorem~\ref{thm:john-ellipsoid}, the polytope $P$ contains the following scaled ellipsoid:
 $$E' \defeq \frac{k}{n}\vones + \left\{  x \in \hk{0} \;:\; \sum_{i=1}^m (\dim V_i) \frac{\| \proj{x}{V_i} \|^2}{\| \proj{z}{V_i} \|^2} \leq \frac{1}{n-1} \right\}.$$
 Since the dimension of $P$ is $n-1$, the scaling factor is $\frac{1}{n-1}$ accordingly (also see Remark~\ref{rem:minimal-enclosing-ellipsoid-dim}).
  Next we derive conditions under which $E'$ and thus also $P$ contain an interior integer point.
 In this case $z$ cannot be a core point.

 Let $u\in\zzk{k} \subset \hk{k}$ be an integer point with minimal norm. 
 If for all subspaces $V_i$ the length of the projection $\norm{\proj{z}{V_i}}$ is large enough, then the following inequality is satisfied.
 \begin{equation}\label{eq:ellipsoid-contains-integer-point}
  \sum_{i=1}^m (\dim V_i) \frac{\norm{ \proj{u}{V_i} }^2}{\| \proj{z}{V_i} \|^2} \leq \frac{1}{n-1}
 \end{equation}
 Hence, in this case the ellipsoid $E'$ contains the integer point $u$.
 Then $u$ must also lie in $P$ by construction of~$E'$.
 For an estimation of when~\eqref{eq:ellipsoid-contains-integer-point} is fulfilled, let $u'\defeq u - \frac{k}{n}\vones$ be the orthogonal projection of $u$ onto $\hk{0}$.
 Because $\norm{\proj{u}{V_i}} \leq \norm{u'}$ and $\dim V_i \leq n-1$,
 inequality~\eqref{eq:ellipsoid-contains-integer-point} is satisfied if for all $i$ the projections satisfy
 \begin{equation}\label{eq:projection-length-global-bound}
  \norm{\proj{z}{V_i}}^2 \geq m(n-1)^2 \norm{u'}^2.
 \end{equation}

 As $u$ was chosen as an integer point in $\hk{k}$ with minimal norm, the bound in~\eqref{eq:projection-length-global-bound} depends only on the layer index $k$ and the dimension~$n$.
 However, since $u + l\vones$ has minimal norm in $\zzk{k+ln}$ for integers $l$, the bound really depends only on the value $k \bmod n$.
 For each $k\in\setOneTo{n}$ we get from~\eqref{eq:projection-length-global-bound} a constant $C(n,k)$ such that: $\norm{\proj{z}{V_i}} \geq C(n,k)$ for all $i$ implies $P$ contains an integer point.
 Since these are only finitely many layers, there exists a constant $C(n) \defeq \max_k C(n,k)$ as claimed in the theorem.
\end{proof}

Theorem~\ref{thm:core-points-close-to-invariant-subspace} remains valid under milder assumptions on~$\Gamma$.
It also holds when $\Gamma \leq \GL_n(\zz)$ is a finite group of unimodular matrices (see~\cite{ThomasPhD}).

\section{Precisely two invariant subspaces -- finitely many core points!}\label{sec:finite-coresets}

In this section we consider groups for which the orthogonal complement of~$\vones$ is irreducible.
Hence, these groups have precisely two invariant subspaces.
Recall that it suffices to study core points up to translation by~$\vones$ (see Remark~\ref{rem:corepoints-up-to-translation}).
It is an immediate consequence of Theorem~\ref{thm:core-points-close-to-invariant-subspace} that the considered groups have only finitely many core points up to translation (see the following Section~\ref{sec:finite-coresets-finiteness}).
Therefore all core points can be enumerated computationally.
In Section~\ref{sec:finite-coresets-outline} we give an overview of our exhaustive search.
The necessary mathematical equipment is provided in Sections~\ref{sec:finite-coresets-bounds} and \ref{sec:finite-coresets-tweaks}.
In Section~\ref{sec:finite-coresets-results} we discuss the results of our computational search for core points.

\subsection{Finiteness}\label{sec:finite-coresets-finiteness}

From Theorem~\ref{thm:core-points-close-to-invariant-subspace} it follows immediately that groups with precisely two invariant subspaces have only a finite number of core points up to translation.
By Cameron~\cite[Lemma~2]{MR0294471}, these are exactly the $2$-homogeneous groups.
\begin{corollary}\label{cor:FiniteCoreSets}
 If $\Gamma \leq \Symmet{n}$ is $2$-homogeneous, then the number of core points up to translation by~$\vones$ is finite.
\end{corollary}
\begin{proof}
 If $\Gamma$ has only two invariant subspaces, then $\rr^n = \lin\vones \oplus V$ for an irreducible invariant subspace~$V$.
 Theorem~\ref{thm:core-points-close-to-invariant-subspace} then shows that every core point must have a ``small'' projection onto~$V$.
 Thus, every core point is contained in a cylinder with radius $C(n)$ around the fixed space $\lin\vones$.
 This cylinder contains only finitely many integral points up to translation by $\vones$.
\end{proof}

We conjecture (see Conjecture~\ref{conj:finite-coreset-iff-group-irreducible}) that the converse statement is true, that is, every transitive group that is not $2$-homogeneous has an infinite number of core points up to translation.
In Section~\ref{sec:infinite-coresets} we will investigate this conjecture more closely.

In the following proposition we estimate the constant $C(n)$ from the proof of Corollary~\ref{cor:FiniteCoreSets}.
To get a good estimate we also consider the dependency on the layer index~$k$.
\begin{proposition}\label{prop:core-set-cylinder}
 Let $\Gamma \leq \Symmet{n}$ be a $2$-homogeneous group.
 For a core point $z\in \zz^n$ with $\scp{z}{\vones} = k$ for $k \in [n-1]$ we have:
 $$\norm{ z - \proj{z}{\lin\vones} } < (n-1)\sqrt{\frac{k(n-k)}{n}}.$$
\end{proposition}
\begin{proof}
 For the proof it is enough to obtain a value for the right hand side of~\eqref{eq:projection-length-global-bound} in the proof of Theorem~\ref{thm:core-points-close-to-invariant-subspace}.
 Let $V$ be the $(n-1)$-dimensional invariant subspace of~$\Gamma$.
 Note that $$\norm{\proj{z}{V}} = \norm{z-\proj{z}{\lin\vones}} = \norm{z - \frac{k}{n}\vones}.$$
 Since $V$ is the only invariant subspace besides the fixed space $\lin\vones$, the value of $m$ in \eqref{eq:projection-length-global-bound} equals one. 
 Therefore we have
 \begin{equation}\label{eq:projection-length-transitive-bound}
 \norm{ z - \frac{k}{n}\vones } = \norm{\proj{z}{V}} < (n-1) \norm{\proj{u}{V}}.
 \end{equation}
 The points $u$ with minimal projected norm in this case are the universal core points from Example~\ref{ex:universal-core-points}.
 For layer~$k$ we can choose $u$ to be any point with $k$ ones and $n-k$ zeros as coordinates.
 We compute
 $$\norm{\proj{u}{V}}^2 = \norm{u}^2 - \norm{\proj{u}{\lin\vones}}^2 = k - \norm{\frac{k}{n}\vones}^2 = k - \frac{k^2}{n} = \frac{k(n-k)}{n}.$$
 Using this value in~\eqref{eq:projection-length-transitive-bound} yields the inequality claimed in the proposition.
\end{proof}

\subsection{On how to determine all core points}\label{sec:finite-coresets-outline}

We now present one way to practically compute all core points of a $2$-homogeneous group up to translation.
Our computational results with this approach will be discussed in Section~\ref{sec:finite-coresets-results} for $2$-homogeneous groups of degree up to twelve.

For the core point enumeration two essential tasks are involved.
First, we need to determine a set of candidates which is large enough to cover all core points, and enumerate its elements up to $\Gamma$-symmetry. 
The quality of the set strongly relies on the quality of the bounds used for the computation.
The bound from Proposition~\ref{prop:core-set-cylinder} is not strong enough in general.
Therefore we use improved bounds that we develop in Section~\ref{sec:finite-coresets-bounds}.
For our core point enumeration we look at the (zero-based) integral points in the cubes $[0,n-3]^n$ for $2$-transitive groups and $[0,\floor{1.09(n-1)}]^n$ for the other $2$-homogeneous groups.
These numbers follow from Theorems~\ref{thm:box-width-bound} and~\ref{thm:box-width-bound-2homog}, respectively.
Note that, by Remark~\ref{rem:corepoints-up-to-translation} and Definition~\ref{def:zero-based}, it is enough to consider only zero-based core points, which leads to the aforementioned cubes.

The second task is to check for each candidate whether it is a core point or not. 
There are two natural ways to tackle this task.
One way is to set up a (mixed) integer program that is feasible if and only if an orbit polytope contains an integral point that is not a vertex (for details see \cite{ThomasPhD}).
Another way is to count the integral points in the orbit polytope and compare it to the number of vertices.
The candidate is a core point if and only if the two numbers coincide.
For our examples we chose the second approach (see Section~\ref{sec:finite-coresets-results}).

Dealing with problems that are NP-hard in general, the second task -- checking whether a candidate is a core point -- is the most time-consuming step in the computation.
Hence, additional criteria are necessary to exclude points from the expensive core point check upfront.
We use the following tweaks, which we will discuss in detail in the next sections.
\begin{itemize}
\item We can assume w.l.o.g. that each candidate $z$ is zero-based (see above) and that its first coordinate is minimal (by transitivity), i.e., $z_1 = 0$.
\item We skip the check for all universal core points. 
 Recall that they are core points with respect to \emph{every} subgroup of~$\Symmet{n}$, compare Example~\ref{ex:universal-core-points}.
 \item For $(k+1)$-transitive groups, Proposition~\ref{prop:core-point-layer-k-transitive} allows for the exclusion of
all integer points with layer index~$l$ for $(l\bmod n)\in\pm\{1,\dots,k\}$ since they
either are universal core points, or their orbit polytope contains one.
 \item All candidates~$z$ whose nonzero coordinates have a greatest common divisor~$\gcd>1$ can be excluded by Lemma~\ref{lem:cp-gcd}.
 \item Let $\Gamma'(z)$ be the stabilizer of the set of even coordinates of a candidate~$z$.
 We skip all candidates which are not constant on the orbits of $\Gamma'(z)$.
 This is justified by Lemma~\ref{lem:even-set-stab-criterion}.
 \item For $2$-transitive groups we check whether 
 $$\left(\sum_{i=1}^n z_i\right) \bmod {(n-1)} \leq \max z_i,$$
 which follows from Proposition~\ref{prop:knoerr-two-transitive}.
 \item Finally, we check whether the orbit polytope $\conv\Gamma z$ contains one of the universal core points (see Section~\ref{subsec:lp-universal-core-points}).
\end{itemize}
We give statistics about the combined power of all these criteria in Table~\ref{tab:coresets-cand-elimination}.

\subsection{Box width bounds}\label{sec:finite-coresets-bounds}

Proposition~\ref{prop:core-set-cylinder} already provides a bound for the distance of a core point~$z\in\hk{k}$ from its projection~$\frac{k}{n}\vones$ onto the fixed space.
This bound turns out to be weak, as shown by our results in Section~\ref{sec:finite-coresets-results}.
It also has the disadvantage that it is not straight-forward to enumerate all integral points inside a ball of a given radius.
In the following we show how to obtain stronger and more practical bounds, which are essential for the viability of 
our computations described in Section~\ref{sec:finite-coresets-results}.
These bounds will be in terms of the \emph{box width}~$\bw(z)$ which we define as
\[\bw(z) := \max_{i \in [n]} z_i - \min_{i \in [n]} z_i.\]
We start with the special case of $2$-transitive groups and come back to the more general case of $2$-homogeneous groups at the end of this section.

\begin{theorem}\label{thm:box-width-bound}
  Let $\Gamma \leq \Symmet{n}$ be a $2$-transitive group and let $n\geq 4$.
  Then $\bw(z) \leq n-3$ for every core point~$z$ with respect to~$\Gamma$.
\end{theorem}
For our proof of this theorem we use the following simple observation about intersection with and projection onto fixed spaces.
Remember that the orthogonal projection of $x$ onto $\fix(\Gamma)$ is given by the barycenter of its orbit:
\begin{equation}\label{eq:projection-fixed-space}
\proj{x}{\fix(\Gamma)} = \frac{1}{\card{\Gamma}} \sum_{\gamma\in\Gamma} \gamma x.
\end{equation}
\begin{lemma}\label{lem:stabilizer-polytope}
  Let $P \subset\rr^n$ be a polytope and $\Gamma\leq\Symmet{n}$ be a symmetry group of~$P$, i.e., $\Gamma P = P$.
  Then
  $$P \cap \fix(\Gamma) = \proj{P}{\fix(\Gamma)}.$$
\end{lemma}
\begin{proof}
 For the ``$\subseteq$''-part let $x \in P \cap \fix(\Gamma)$.
 Since $x \in \fix(\Gamma)$, we have $x = \proj{x}{\fix(\Gamma)} \in \proj{P}{\fix(\Gamma)}$.
 For the reverse inclusion ``$\supseteq$'' let $y \in \proj{P}{\fix(\Gamma)}$.
 In particular, $y=\proj{x}{\fix(\Gamma)}$ is a convex combination of points $\gamma x$ in $\Gamma P$ by~\eqref{eq:projection-fixed-space}.
 Since $\Gamma P = P$ is convex, this implies that the point~$y$ lies in $P$.
\end{proof}
In words, the lemma states that projection to the fixed space equals intersection with the fixed space for symmetric polytopes.
Depending on how a polytope is presented, either by facets or by vertices, one of these two operations is easier to handle.
Since we are dealing with orbit polytopes, we naturally only have its vertices, so the projection is readily available.
Lemma~\ref{lem:stabilizer-polytope} allows us to find integral points in $P$, which may be difficult, 
by finding projections of integral points, which may be a much easier problem.
How easy it gets depends on the group we choose.
Consider an orbit polytope $\conv \Gamma z$.
If we intersect it with the fixed space $\fix(\Gamma)$, this leaves us with the vertex barycenter of the orbit polytope, which does not provide new information.
Thus, the goal is to find a subgroup $\Gamma' \lneq \Gamma$ with at least two but still a small number of orbits so that the projection is not trivial.
In particular, for $2$-transitive groups, which we focus on in this section, we can obtain a one-dimensional projection, using a subgroup with two orbits.
In such a line segment integral points are naturally easy to find.
Theorem~\ref{thm:box-width-bound} follows from the fact that if the line segment is wide enough, it -- and therefore also the original polytope -- contain an integer point, which is not a vertex.
To prove the main theorem, we start with an application of Lemma~\ref{lem:stabilizer-polytope} to $2$-transitive groups.

\begin{proposition}\label{prop:two-transitive-polytope}
	Let $\Gamma \leq \Symmet{n}$ be a $2$-transitive group and let $P \defeq \conv \Gamma z$ be the orbit polytope of some zero-based $z \in \zzp^n$.
	Then a point $p = (k,l,l,\dots,l)^\tp\in \rr^n$ for some $k,l \in \rr$ lies in $P$ if and only if the following two conditions are met:
	\begin{enumerate}[ref=(\roman{*}),label=(\roman{*})]
		\item $0 \leq k \leq \max z_i$,
		\item $l = \frac{(\sum_{j=1}^n z_j) - k}{n-1}$.
	\end{enumerate}
\end{proposition}
\begin{proof}
	The stabilizer $\Gamma'\defeq\stab_\Gamma(p) = \stab_\Gamma(1)$ of $p$ acts transitively on $\{2,\dots,n\}$ because $\Gamma$ is $2$-transitive.
	Let $\{\gamma_1, \dots, \gamma_n\}\subset \Gamma$ be a transversal for $\Gamma$ modulo $\Gamma'$, that is, $\gamma_i(i)=1$ for each $i \in [n]$.
	Thus, for every $\gamma_i$ we have that
	$$\proj{({\gamma_i}z)}{\fix(\Gamma')} = (z_i,r_i,r_i,\dots,r_i)^\tp \qquad \text{where} \quad r_i = \frac{1}{n-1} \sum_{j \in [n]\setminus\{i\}} z_j.$$
	Let $Q \defeq \proj{P}{\fix(\Gamma')}$ be the projection of $P$ onto the fixed space $\fix(\Gamma')$.
	It is the convex hull of vectors $\cin{q}{i} \defeq \proj{({\gamma_i}z)}{\fix(\Gamma')}$ for $i\in[n]$.
	All these vectors lie in a one-dimensional affine subspace of $\rr^n$, so~$Q$ is a line segment.
  By Lemma~\ref{lem:stabilizer-polytope} the point $p \in\fix(\Gamma')$ lies in $P$ if and only if it lies in the projection~$Q$.
  
	Let $a$ be such that $z_a = \min_{i \in [n]} z_i = 0$ and let $b$ be such that $z_b = \max_{i \in [n]} z_i$.
	With this setting we know that $\cin{q}{a}$ and $\cin{q}{b}$ are end points of $Q$ because of the respective minimality and maximality of $z_a$ and $z_b$.
  To simplify notation we project on the first two coordinates, which are sufficient.
  We identify $Q$ with the line segment $Q' \subset \rr^2$, given as the convex hull of $\cin{q'}{a} = (0, r_a)^\tp$ and $\cin{q'}{b}=(z_b, r_b)^\tp$.
  As inequality description we obtain $$Q' = \left \{ (x_1,x_2)^\tp \in \rr^2 \;:\; 0 \leq x_1 \leq z_b \quad\text{and}\quad x_1 + (n-1)x_2 = \sum_{j=1}^n z_j \right \}.$$
	Hence, the polytope $Q'$	contains a point $u=(u_1, u_2)^\tp \in \rr^2$ if and only if $0 \leq u_1 \leq z_b$ and $u_2 = \frac{1}{n-1} (\sum_{j=1}^n z_j) - \frac{u_1}{n-1}$.
	Because the point $p$ of the proposition projects onto $(k,l)\in \rr^2$, the claim of the proposition follows.
\end{proof}

A simple observation for which points cannot be core points is the following lemma.

\begin{lemma}\label{lem:cp-gcd}
 Let $z\in \zz^n$ be a 
 core point for a group $\Gamma\leq \Symmet{n}$.
 If $z \notin \fix(\Gamma)$, then $\gcd(z_1,\dots,z_n) = 1$.
\end{lemma}
\begin{proof}
 Let $z \in \zz^n$ 
 have $c\defeq\gcd(z_1,\dots,z_n) > 1$.
 In order to prove the lemma we show that such a point~$z$ is not a core point.
 Because $z \notin \fix(\Gamma)$ by assumption of the lemma, there is a permutation $\gamma\in \Gamma$ with $\gamma z \neq z$.
 Then $\frac{c-1}{c} z + \frac{1}{c} \gamma z$ is an integral, non-trivial convex combination of two vertices of $\conv \Gamma z$.
 Hence, $\conv \Gamma z$ is not lattice-free and $z$ is not a core point.
\end{proof}

Theorem~\ref{thm:box-width-bound} will follow from the following proposition.
The previous Proposition~\ref{prop:two-transitive-polytope} showed that we can find integral points in a polytope by finding integral points on a line segment in $\rr^2$ with slope $(n-1):1$.
The following proposition due to Knörr~\cite{KnoerrSketch} quantifies the condition under which the induced line segment contains an integral point.
It is also interesting in its own right because it states a necessary criterion for core points which is stronger than the box width alone.

\begin{proposition}\label{prop:knoerr-two-transitive}
 Let $\Gamma \leq \Symmet{n}$ be a $2$-transitive group with $n\geq 3$.
 Let $z \in \zzp^n$ be zero-based with $\max z_i \geq 2$.
 If $$\left(\sum_{i=1}^n z_i\right) \bmod {(n-1)} \leq \max z_i,$$ then $\conv \Gamma z$ is not lattice-free.
\end{proposition}
\begin{proof}
 Let $k\in \{0,\ldots, n-2\}$ be congruent to $\sum_{i=1}^n z_i \bmod {(n-1)}$.
 Then $l \defeq \frac{(\sum_{i=1}^n z_i) - k}{n-1}$ is an integer.
 By Proposition~\ref{prop:two-transitive-polytope} the integral point $p = (k,l,\dots,l)^\tp$ lies in $P\defeq\conv \Gamma z$ because $0 \leq k \leq \max z_i$.
 The point $p$ is a vertex of $P$ if and only if $p$ is in the orbit of~$z$.
 If $p$ is not a vertex, then $P$ is not lattice-free and we are done.
 So suppose that $p$ is a vertex of $P$.
 Because $z$ is zero-based, this can happen only if
 $l = 0$ or $k = 0$.
 In these two cases we still have to find an integer point in $P$ which is not a vertex.
 Note that in both cases we must have $\gcd(p_1,\dots,p_n) = \gcd(k,l) \geq 2$ because of our assumption $\max z_i \geq 2$.
 Thus, Lemma~\ref{lem:cp-gcd} implies that $\conv \Gamma p$ is not lattice-free and therefore $\conv \Gamma z \supseteq \conv \Gamma p$ is not lattice-free. 
\end{proof}

With this proposition we are able to prove the maximal box width of core points for $2$-transitive groups.
\begin{proof}[Proof of Theorem~\ref{thm:box-width-bound}]
  It suffices to prove the theorem for zero-based points because the box width is not affected by translation by $\vones$.
  Let $z \in \zzp^n$ be zero-based with $\max z_i \geq n-2 \geq 2$.
  We have to show that $z$ is not a core point.
  It holds that $\sum_{i=1}^n z_i \bmod {(n-1)} \leq \max z_i$ because the remainder of $\sum_{i=1}^n z_i$ after division by $n-1$ lies in $\{0,1,\dots,n-2\}$.
  Thus, Proposition~\ref{prop:knoerr-two-transitive} ensures that the orbit polytope $\conv \Gamma z$ is not lattice-free.
  Hence, $z$ is not a core point and the claim of the theorem follows.
\end{proof}

For $2$-homogeneous groups the situation is more complicated than for the $2$-transitive groups.
We can start similarly and study the projection of orbit polytopes onto the fixed space $\fix(\stab_\Gamma(1))$.
Because this fixed space has dimension three (see~\cite[Lemma~2]{MR0294471}), the resulting projected polytope is in general not a line segment but a two-dimensional polygon.
For two-dimensional polytopes, determining the vertices and integral points is not as trivial as in the one-dimensional case.
Using the classification of $2$-homogeneous, not $2$-transitive permutation groups (see~\cite{MR0306296}) and the flatness theorem in dimension two (see~\cite{MR1060014}), one can still obtain the following upper bound on the box width.
Its proof is quite technical so we just state the result here and refer to~\cite{ThomasPhD} for details.
\begin{theorem}[\cite{ThomasPhD}]\label{thm:box-width-bound-2homog}
 Let $\Gamma \leq \Symmet{n}$ be a $2$-homogeneous group.
 Then $\bw(z) < 1.09\,(n-1)$ for every core point~$z$ with respect to~$\Gamma$.
\end{theorem}

\subsection{Tweaks to speed up computations}\label{sec:finite-coresets-tweaks}

\subsubsection{A parity tweak}

\begin{lemma}\label{lem:even-set-stab-criterion}
  Let $\Gamma \leq \Symmet{n}$ be a transitive permutation group and $z\in \Z^n$.
  Consider the set $E$ of indices corresponding to the even coordinates of~$z$, that is, $E = \{i\in\setOneTo{n} \;:\; z_i \equiv 0 \bmod 2\}$.
  Let $\Gamma'(z)\leq\Gamma$ be the set-stabilizer of~$E$.
  Further, let~$\mathcal{I}$ be the partition of $[n]$ into orbits
  under~$\Gamma'(z)$. If any of the orbits $O\in\mathcal{I}$ contains two indices
  $k,l\in O$ such that~$z_k$ is not equal to~$z_l$, then the point~$z$ is not a core point with respect to $\Gamma$.
\end{lemma}
\begin{proof}
  Let~$k, l\in O$ be two indices in the same orbit $O \in \mathcal{I}$ with $z_k\neq
  z_l$.
  Since~$\Gamma'(z)$ acts transitively on $O$, there exists a
  permutation~$\gamma\in\Gamma'(z)$ such that the $l$-th coordinate of $\gamma
  z$ is equal to~$z_k\neq z_l$. Since the coordinates corresponding to indices in every orbit in $\mathcal I$ are either all even or all odd, the point
  $z'\coloneqq\frac{1}{2}z+\frac{1}{2}\gamma z$ is integral. Furthermore, it is a proper
  convex combination, as $z\neq\gamma z$. Hence, the integer point~$z'$ is contained in
  the orbit polytope of~$z$ without being a vertex, thus~$z$ is not a core point.
\end{proof}

Note that Lemma~\ref{lem:even-set-stab-criterion} also holds with respect to odd instead of even coordinates.
In order to use the lemma in the candidate enumeration, it is necessary to compute the orbits of the set-stabilizers of all subsets of $[n]$ up to $\Gamma$-symmetry in a preprocessing step.
However, using a software package like~\cite{GAP} the computation time for this task is negligibly small.
The lemma is particularly effective if the set stabilizers have large orbits (so that many coordinates must have the same value).
By a result of Seress~\cite{MR1468057}, many small $2$-transitive groups are exceptional in the sense that no set stabilizer is trivial, i.e., has at least one orbit of size two (and usually many more).

\subsubsection{Restriction on layer indices for $k$-transitive groups}

We can generalize the argument behind Proposition~\ref{prop:two-transitive-polytope} to groups of higher transitivity.
The following proposition shows that transitivity enforces that core points with ``small'' layer index~$k$ must be universal core points.
For the enumeration in Section~\ref{sec:finite-coresets-results} we can thus skip these layers.

\begin{proposition}\label{prop:core-point-layer-k-transitive}
Let $\Gamma \leq \Symmet{n}$ be a $(k+1)$-transitive group with $k\geq 1$.
Then the only core points with respect to $\Gamma$ in $\zzk{l}$, for $l \bmod n$ congruent to an index in $\{0,\dots,k\}\cup\{n-k,\dots,n\}$, 
are universal core points.
\end{proposition}
\begin{proof}
 Let $z \in \zzp^n$ be zero-based and $\max z_i \geq 2$, otherwise $z$ is universal.
 To keep the index notation simple we may assume that $z$ is sorted non-decreasingly.
 If $z$ is not already sorted, we relabel the coordinates.
 Further, we write $N$ for the layer index $N=N(z)\defeq\scp{\vones}{z} = \sum_{i=1}^n z_i$.
 To prove the proposition, it is enough to show that every such $z$ with $N \equiv k \bmod n $ is not a core point because every $(k+1)$-transitive group is $k$-transitive.
 In the following we prove that $P\defeq\conv \Gamma z$ is not lattice-free by using Lemma~\ref{lem:stabilizer-polytope}.
 More precisely, we show that $P$ contains 
 \begin{equation}\label{eq:special-vector-v}
 v = (\underbrace{c+1,\dots,c+1}_{k~\text{times}},\underbrace{c,\dots,c}_{n-k~\text{times}})
 \end{equation}
 for $c = \left\lfloor \frac{N}{n} \right\rfloor$.
 Note that $v$ is contained in the fixed space $\fix(\Gamma')$ of the set stabilizer $\Gamma' \defeq \stab_\Gamma(\{1,\dots,k\})$.
 By Lemma~\ref{lem:stabilizer-polytope} it suffices to prove that $v$ is contained in the projection $Q\defeq\proj{P}{\fix(\Gamma')}$ in order to ensure that $v$ lies in $P$.
 
Because the group $\Gamma$ is $(k+1)$-transitive, the stabilizer $\Gamma'$ acts transitively on the sets $\{1,\dots,k\}$ and $\{k+1,\dots,n\}$.
Thus, the projection of an $x$ onto the fixed space is given by $\proj{x}{\fix(\Gamma')} = (R(x),\dots,R(x),S(x),\dots,S(x))^\tp$ with $R(x) \defeq \frac{1}{k}\sum_{i=1}^k x_i$ and $S(x) \defeq \frac{1}{n-k}\sum_{i=k+1}^n x_i$.
Therefore, $Q$ is a line-segment that is contained in the hyperplane~$\hk{N} = \{x\in\rr^n\;:\;\scp{\vones}{x}=N\}$.
In the following we show the existence of two points $x, y \in Q$ with $R(x) \leq R(v) \leq R(y)$.
By our initial assumption we have $N = cn+k = (n-k)c + k(c+1)$ and thus $v\in\hk{N}$.
Hence, the existence of such $x$ and $y$ implies that $v$ lies on the line-segment~$Q$.
Our next step is to show that
\begin{align}
 \label{eq:sum-first-k} \sum_{i=1}^k z_i &\leq k(c+1)\qquad{\text{and}}\\
 \label{eq:sum-last-k} \sum_{i=n-k+1}^n z_i &\geq k(c+1).
\end{align}
After we have established these inequalities, we immediately obtain the desired points $x$ and $y$ as follows.
From the first equation~\eqref{eq:sum-first-k} we get that $R(z) \leq (c+1) = R(v)$.
Because $\Gamma$ is $k$-transitive, there is a permutation $\gamma\in\Gamma$ that maps $\{n-k+1,\dots,n\}$ to $\{1,\dots,k\}$.
Thus, we obtain $R(v) = (c+1) \leq R(\gamma z)$ from~\eqref{eq:sum-last-k}.
This shows that the choice $x = z$ and $y = \gamma z$ satisfies our requirements.

It remains to show that inequalities \eqref{eq:sum-first-k} and \eqref{eq:sum-last-k} actually hold.
For a contradiction assume that $\sum_{i=1}^k z_i > k(c+1)$.
Since $z$ is sorted, this implies $z_k \geq c+2$ and thus $$N = \sum_{i=1}^n z_i = \sum_{i=1}^k z_i + \sum_{i=k+1}^n z_i > k(c+1) + (n-k)(c+2) > N.$$
We get a similar contradiction by assuming that $\sum_{i=n-k+1}^n z_i < k (c+1)$.
This implies $z_{n-k+1} \leq c$ and thus $$N = \sum_{i=1}^n z_i = \sum_{i=1}^{n-k} z_i + \sum_{i=n-k+1}^n z_i  < (n-k)c + k(c+1) = N.$$
Therefore the inequalities \eqref{eq:sum-first-k} and \eqref{eq:sum-last-k} must hold.

Thus, we have shown that $v \in Q$ and therefore also $v \in P$.
We still have to prove that $v$ is not a vertex of $P$, i.e., $v$ is not in the orbit of $z$.
Because $z$ is zero-based, the point~$v$ can only be a vertex of~$P$ if $c=0$.
Otherwise, all coordinates of $v$ are non-zero by choice of $v$ in~\eqref{eq:special-vector-v}.
So we can assume that $c=0$.
In this case we have $\max z_i = c+1 = 1$, which we have ruled out by our initial assumption.
Hence, $v$ is not a vertex of~$P$.
\end{proof}

A simple corollary of this proposition is the following.
However, a similar statement for general $2$-homogeneous groups is false as the computer search in Section~\ref{sec:finite-coresets-results} shows.
\begin{corollary}
If $\Omega \leq \Symmet{n}$ is $2$-transitive, then all core points in the layers with index $1$ and $n-1$ are universal core points.
\end{corollary}

\subsubsection{Selective probing}\label{subsec:lp-universal-core-points}

If none of the other, easily testable criteria excluded a given candidate, we apply a last heuristic before we call the computationally expensive lattice point enumeration.
For this we choose a selection of \glqq probing points\grqq\ to check wether a 
given candidate $z$ is a core point. 
For each point $z$ which is not a core point the orbit polytope $\conv \Gamma z$ is likely to contain one of the core points already approved.
The check whether a specific point is contained in $\conv\Gamma z$ can be done by solving one linear program, see for instance \cite{PolyFAQ}.
In order to keep the computational effort within reasonable limits,
it is of course advisable to choose only a selection of already approved core points, 
if there are too many.
For our enumeration, for instance, we checked
whether the orbit polytope of a candidate contains one of the universal core points.
These are a natural choice since they are the closest points to the vertex barycenter of the orbit polytope.
Probing for them enabled us to eliminate a large number of candidates (see Table~\ref{tab:coresets-cand-elimination}).

\subsection{Computational results}\label{sec:finite-coresets-results}

We now present the results of our exhaustive computer search based on the strategy described in the previous sections.
To enumerate all core points of all $2$-homogeneous groups with degree up to twelve we  implemented the core point enumeration using the polymake framework~\cite{polymake,polymake-2000}.
An overview of these groups is shown in Table~\ref{tab:coresets-irred-groups}.
The column ``Id'' is a composition of the group degree and the \texttt{PrimitiveIdentification}-id of the group as assigned by the library of primitive groups of~\cite{GAP}.

\begin{table}[htb]
\caption{$2$-homogeneous groups up to degree 12}
  \label{tab:coresets-irred-groups}
  \renewcommand{\arraystretch}{0.9}
  \begin{tabular*}{\linewidth}{@{\extracolsep{\fill}}rrrrr@{}}
    \toprule
    Id & Order & Structure & Transitivity & Homogeneity\\
    \midrule
    5-3 & 20& $\AGL(1,5)$ & 2 & 5  \\
\midrule
6-1 & 60& $\PSL(2,5)$& 2 & 2  \\
6-2 & 120& $\PGL(2,5)$& 3 & 6 \\
\midrule
7-3 & 21& $\Cycl{7}\rtimes\Cycl{3}$& 1 & 2  \\
7-4 & 42& $\AGL(1, 7)$& 2 & 2  \\
7-5 & 168& $\LL(3, 2)$& 2 & 2  \\
\midrule
8-1 & 56& $\AGL(1, 8)$& 2 & 3  \\
8-2 & 168& $\AGGL(1, 8)$& 2 & 3  \\
8-3 & 1344& $\ASL(3, 2)$& 3 & 3  \\
8-4 & 168& $\PSL(2, 7)$& 2 & 3  \\
8-5 & 336& $\PGL(2, 7)$& 3 & 3  \\
\midrule
9-3 & 72& $M_9$& 2 & 2  \\
9-4 & 72& $\AGL(1,9)$& 2 & 2  \\
9-5 & 144& $\AGGL(1, 9)$& 2 & 2  \\
9-6 & 216& $3^2$:(2'A(4))& 2 & 2  \\
9-7 & 432& $\AGL(2, 3)$& 2 & 2  \\
9-8 & 504& $\PSL(2, 8)$& 3 & 9  \\
9-9 & 1512& $\PGGL(2, 8)$& 3 & 9  \\
\midrule
10-3 & 360& $\PSL(2, 9)$& 2 & 2  \\
10-4 & 720& $\PGL(2, 9)$& 3 & 3 \\
10-5 & 720& $\Symmet{6}$& 2 & 2  \\
10-6 & 720& $M_{10}$& 3 & 3  \\
10-7 & 1440& $\PGGL(2, 9)$& 3 & 3  \\
\midrule
11-3 & 55& $\Cycl{11}\rtimes\Cycl{5}$& 1 & 2  \\
11-4 & 110& $\AGL(1, 11)$& 2 & 2  \\
11-5 & 660& $\LL(2, 11)$& 2 & 2  \\
11-6 & 7920& $M_{11}$& 4 & 4  \\
\midrule
12-1 & 7920& $M_{11}$& 3 & 3  \\
12-2 & 95040& $M_{12}$& 5 & 5  \\
12-3 & 660& $\PSL(2, 11)$& 2 & 3  \\
12-4 & 1320& $\PGL(2, 11)$& 3 & 3  \\

    \bottomrule
  \end{tabular*}
\end{table}
Regarding the core point search, Table~\ref{tab:theoretical-bounds} shows that there is a vast number of core point candidates in the cube induced by our theoretical bound on the box width.
\begin{table}[htb]
\begin{threeparttable}
\caption{Theoretical maximal bounds for $2$-transitive groups}\label{tab:theoretical-bounds}
 \begin{tabular*}{\linewidth}{@{\extracolsep{\fill}}rrrr@{}}
 \toprule
   dim $n$ & \#integral points in $[0,\bw]^n$ & $\bw$ & distance to $\lin\vones$\\
   \midrule
5 & 243 & 2 & 4.38\\
6 & 4\,096 & 3 & 6.12\\
7 & 78\,125 & 4 & 7.86\\
\tnote{a}7 & 823\,543 & 6 & 7.86\\
8 & 1\,679\,616 & 5 & 9.90\\
9 & 40\,353\,607 & 6 & 11.93\\
10 & 1\,073\,741\,824 & 7 & 14.23\\
11 & 31\,381\,059\,609 & 8 & 16.51\\
\tnote{a}11 & 285\,311\,670\,611 & 10 & 16.51\\
12 & 1\,000\,000\,000\,000 & 9 & 19.05\\
\bottomrule
 \end{tabular*}
 \begin{tablenotes}
   \small
   \item [a] for the $2$-homogeneous case, for which the larger $\bw$-bound applies
  \end{tablenotes}
\end{threeparttable}
\end{table}

Table~\ref{tab:coresets-cand-elimination} illustrates the progress of our candidate elimination towards the actual set of core points.
\begin{table}[ht]
\begin{threeparttable}
 \caption{Candidate elimination}\label{tab:coresets-cand-elimination}
 \renewcommand{\arraystretch}{0.9}
  \begin{tabular*}{\linewidth}{@{\extracolsep{\fill}}rrrrrr@{}}
    \toprule
group id & tweaks & probing & core points & max $\bw$ & max dist to $\lin\vones$\\
    \midrule
    5-3 & 0 & 0 & 0 & -- & -- \\
\midrule
6-1 & 0 & 0 & 0 & -- & -- \\
6-2 & 0 & 0 & 0 & -- & -- \\
\midrule
7-3 & 63\,077 & 12 & 10 & 3 & 2.62\\
7-4 & 10 & 1 & 1 & 2 & 1.85\\
7-5 & 3 & 2 & 2 & 2 & 1.93\\
\midrule
8-1 & 1\,797 & 4 & 4 & 2 & 1.97\\
8-2 & 20 & 1 & 1 & 2 & 1.97\\
8-3 & 3 & 1 & 1 & 2 & 1.97\\
8-4 & 10 & 2 & 2 & 2 & 1.97\\
8-5 & 2 & 0 & 0 & -- & -- \\
\midrule
9-3 & 21\,666 & 20 & 20 & 3 & 2.75 \\
9-4 & 21\,691 & 20 & 18 & 3 & 2.75 \\
9-5 & 529 & 10 & 10 & 3 & 2.75 \\
9-6 & 68 & 3 & 3 & 2 & 2.05 \\
9-7 & 32 & 3 & 3 & 2 & 2.05 \\
9-8 & 5 & 0 & 0 & -- & -- \\
9-9 & 5 & 0 & 0 & -- & -- \\
\midrule
10-3 & 514 & 8 & 8 & 2 & 2.37 \\
10-4 & 31 & 2 & 2 & 2 & 2.12 \\
10-5 & 164 & 6 & 6 & 2 & 2.37 \\
10-6 & 53 & 4 & 4 & 2 & 2.12 \\
10-7 & 31 & 2 & 2 & 2 & 2.12 \\
\midrule
11-3 & \tnote{a} 266\,982 & 2\,546 & 2\,407 & 6 & 5.80 \\
11-4 & 9\,352\,389 & 231 & 208 & 4 & 3.77 \\
11-5 & 4\,285 & 11 & 11 & 2 & 2.76 \\
11-6 & 16 & 2 & 2 & 2 & 2.17 \\
\midrule
12-1 & 128 & 4 & 4 & 2 & 2.58 \\
12-2 & 11 & 1 & 1 & 2 & 2.22 \\
12-3 & 21\,580\,154 & 15 & 15 & 4 & 3.30 \\
12-4 & 7\,252 & 2 & 2 & 2 & 2.22 \\

    \bottomrule
  \end{tabular*}
  \begin{tablenotes}
   \small
   \item [a] number after two-dimensional IPs; see text for an explanation
  \end{tablenotes}
\end{threeparttable}
\end{table}
We introduce the table by columns.
The first column shows the id of the group; this is the same as in Table~\ref{tab:coresets-irred-groups}.
The second column ``tweaks'' gives the number of all actually enumerated candidates, using all necessary bounds and tweaks from Sections~\ref{sec:finite-coresets-bounds} and \ref{sec:finite-coresets-tweaks}, but without selective probing.
The number of candidates shown in this column is much smaller than the number of integral points in the cube $[0,n-3]^n$ for $2$-transitive groups and $[0,\floor{1.09(n-1)}]^n$ for $2$-homogeneous, not $2$-transitive groups (cf. Table~\ref{tab:theoretical-bounds}).
This demonstrates the combined power of all the ``small'' necessary criteria displayed above.
For the groups for which at least one set stabilizer is trivial (7\nobreakdash-3, 9\nobreakdash-3, 9\nobreakdash-4, 11\nobreakdash-3, 11\nobreakdash-4, 12\nobreakdash-3) the number of candidates is much higher than for the other groups (cf.~\cite{MR1468057}).

Note that among the considered groups there are two that are $2$-homogeneous but not $2$-transitive.
These occur in dimension seven and eleven only (groups 7-3 and 11-3).
To these groups we cannot apply Proposition~\ref{prop:knoerr-two-transitive} to eliminate candidates.
For the group with id 11-3, this leaves us with 1\,331\,476\,291 candidates.
This number is too large to proceed to the actual core point checks.
To exclude candidates fast we implemented a test based on Lemma~\ref{lem:stabilizer-polytope}.
We project each orbit polytope $P$ onto the three-dimensional fixed space $\fix(\stab_\Gamma(1))$.
The corresponding projected polytope $Q$ is two-dimensional.
We can find integral points in $Q$ quickly after a relatively cheap convex hull computation.
As integer points in $Q$ correspond to integer points in $P$, this allows to eliminate core point candidates without constructing the complete orbit polytope.
This reduced the number of candidates to under 300\,000 without too much computational effort.
More details can be found in \cite{ThomasPhD}.

The third column of Table~\ref{tab:coresets-cand-elimination} ``probing'' shows the number of candidates that remain after selective probing, that is, the number of orbit polytopes that do not contain a universal core point.
The fourth column ``core points'' lists the number of actual non-universal core points as confirmed by actually enumerating all integral points in the orbit polytopes.
For this final check we use Normaliz \cite{normaliz,bis-2012} via its interface to polymake.
Comparing the third and fourth columns of Table~\ref{tab:coresets-cand-elimination}, we see that the number of candidates after selective probing is already very close to the number of actual non-universal core points.
This shows that a concise description of those that survive probing, i.e., of those points whose orbit polytopes do not contain universal core points, would probably make core point enumeration much easier.

The fifth column of Table~\ref{tab:coresets-cand-elimination} ``max $\bw$'' contains the maximal box width of a core point.
The sixth column ``max dist to $\lin\vones$'' shows the maximal distance of a core point from the fixed space.
Comparing these last two columns to the last two columns of Table~\ref{tab:theoretical-bounds}, we see that the bounds from Theorems~\ref{thm:box-width-bound} and~\ref{thm:box-width-bound-2homog} (for the box width) and Proposition~\ref{prop:core-set-cylinder} (for the cylinder radius) have room for improvement.

The polytopes of all core points of the groups from Table~\ref{tab:coresets-irred-groups} are available in the polymake-format at
\begin{center}
\url{http://www.polymake.org/polytopes/core-point-polytopes/}.
\end{center}

\section{More than two invariant subspaces -- infinitely many core points?}\label{sec:infinite-coresets}

In the previous section we showed that $2$-homogeneous groups (those with precisely two invariant subspaces) have a finite number of core points (up to translation by~$\vones$).
For other groups there may be an infinite number of core points.
For instance, for all integers $m\in\zz$ the point $(1+m,-m,m,-m)^\tp$ is a core point of the cyclic group $\Cycl{4}$ as we will see later (cf.~Example~\ref{ex:invariant-subspace-C6-continued}).
Figure~\ref{fig:orbit-polytope-C4} visualizes parts of this infinite sequence, showing orthogonal projections of the lattice-free orbit tetrahedra for $0\leq m\leq 4$.
\begin{figure}[htb]
\begin{center}
\includegraphics[scale=0.19]{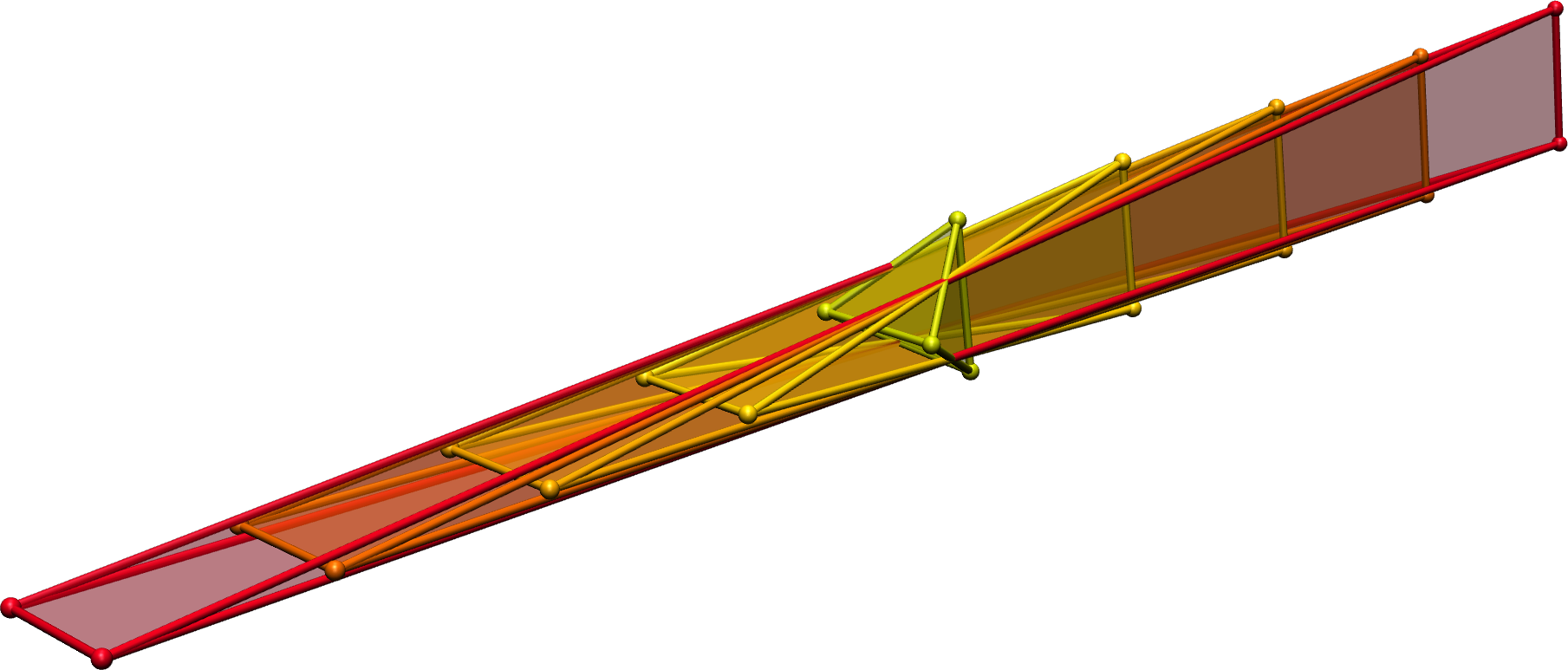}
\end{center}
\caption{Impression of an infinite sequence of lattice-free orbit polytopes for $\Cycl{4}$}\label{fig:orbit-polytope-C4}
\end{figure}
In this section we construct similar infinite sequences of core points (up to translation by~$\vones$) for two major classes of groups.
These constructions, together with our computational experiments, suggest the following conjecture.
\begin{conjecture}\label{conj:finite-coreset-iff-group-irreducible}
  A transitive permutation group $\Gamma$ has a finite number of core points up to translation~by $\vones$ if and only if $\Gamma$ is $2$-homogeneous.
\end{conjecture}

For the aforementioned core point constructions we use the fact that all core points are close to an invariant subspace of the group by Theorem~\ref{thm:core-points-close-to-invariant-subspace}.
In Section~\ref{sec:infinite-construction-idea} we will look at an outline of a general core point construction based on proximity to invariant subspaces.
We use this method to give constructions for all imprimitive groups (in Section~\ref{sec:infinite-rational-subspace}) and for all groups with a non-rationally generated invariant subspace (in Section~\ref{sec:infinite-irrational-subspace}).
For the remaining groups the construction can not be applied directly in general.
However, we computationally verified Conjecture~\ref{conj:finite-coreset-iff-group-irreducible} for all transitive groups up to degree 127.
Details about these special constructions can be found in \cite{KatrinPhD, ThomasPhD}.

\subsection{Constructing core points along invariant subspaces}\label{sec:infinite-construction-idea}

Our main tool in this section is orthogonal projection to an arbitrary invariant subspace of a transitive group.
If this projection of an integer point $z$ has small norm, i.e., the point $z$ is close to an invariant subspace, then $z$ seems to be a good candidate for a core point.
Recall from Section~\ref{sec:permutation-group-basics} that we can always decompose $\rr^n$ into a direct sum of pairwise orthogonal invariant subspaces $\rr^n = \lin \vones \oplus \bigoplus_i V_i$.
If such an invariant subspace $V_i$ contains no rational vectors, i.e., $V_i \cap \qq^n = \{0\}$, we call $V_i$ an \defterm{irrational invariant subspace}.
Similarly, we say that $V_i$ is \defterm{rational} if it has a rational basis.
Every $\rr$-irreducible invariant subspace is either rational or irrational.
Reducible subspaces may be neither rational nor irrational by this definition, but for our purposes it is enough to cover irreducible subspaces.
For some groups, for instance, cyclic groups of prime order, all irreducible invariant subspaces except the fixed space are irrational.
A more detailed study of these groups can be found in \cite{MR2150939}.

Our goal throughout this section is the construction of core points.
Therefore we need a way to prove that an orbit polytope $\conv \ac{z}{\Gamma}$ is lattice-free.
The main tool that we use is projection onto an invariant subspace of $\Gamma$.
If both the projection and the fibers are lattice-free in some sense, then we can prove lattice-freeness for the whole orbit polytope.
Proposition~\ref{prop:fullorbit-infinite-coresets} will give a sufficient core point condition in quite general (and also quite technical) terms.
The rest of this section makes the projection argument more precise.

An important property of the projection to an invariant subspace is that group action and projection operation commute:
\begin{lemma}\label{lem:projection-invariant-subspace-commute-group-action}
 Let $\Gamma\leq\Symmet{n}$ be a permutation group and $V$ an invariant subspace of $\Gamma$.
 Group action and projection commute: $\proj{(\gamma x)}{V} = \gamma (\proj{x}{V})$ for all $\gamma \in \Gamma$ and $x\in\rr^n$.
\end{lemma}
\begin{proof}
 Let $W \defeq V^\perp$ be the orthogonal complement of $V$.
 We can decompose $x = v \oplus w$ with $v\in V$ and $w\in W$ into a direct sum from distinct invariant subspaces $V, W$.
 Since the action of $\Gamma$ is linear, we have $\gamma x = \gamma v + \gamma w$ for every permutation $\gamma \in \Gamma$.
 Because $V$ and $W$ are invariant subspaces, we must have $\gamma v \in V$ and $\gamma w \in W$.
 Hence, this is a direct sum $\gamma x = \gamma v \oplus \gamma w$.
 Thus, $\proj{(\gamma x)}{V} = \gamma v = \gamma (\proj{x}{V})$.
\end{proof}

We now turn to a method for proving lattice-freeness of orbit polytopes.
Let $\Gamma\leq\Symmet{n}$ be a permutation group and $V$ be an invariant subspace of~$\Gamma$.
Furthermore, let $z\in\zzk{k}$ be an integral point in the $k$-th layer.
Since all integer points in the orbit polytope $P \defeq \conv \Gamma z$ also lie in the $k$-th layer, we start with the following projection setup.
We project both the orbit polytope~$P$ and all integer points $\zzk{k}$ orthogonally onto~$V$.
To ensure the lattice-freeness of~$P$ we have to control the pre-image of all points in the intersection $Q\defeq \proj{P}{V} \cap \proj{\zzk{k}}{V}$.
If the pre-image of~$Q$ intersects~$P$ only at its vertices $\vertices(P)$, then $P$ is lattice-free.
This condition is in general quite hard to test because it is an integer feasibility problem.
Thus, we use relaxed conditions instead.
The following two steps together allow us to control the pre-images of $Q$ in some cases.
First, we ensure that all integer points in $P$ project only onto $\proj{\vertices(P)}{V}$.
Second, we ensure that only vertices of $P$ project onto $\proj{\vertices(P)}{V}$.
These two steps together constitute Proposition~\ref{prop:fullorbit-infinite-coresets}.
Before we get there, we start with an outline that introduces facts and notation.

For the first step we use arguments based on the Euclidean norm.
We say that $z$ has \defterm{globally minimal} projection onto~$V$ if
\begin{equation}\label{eq:projection-globally-minimal}
\norm{\proj{z}{V}} \leq \norm{\proj{z'}{V}}\quad\text{for all $z' \in \zzk{k}$,}
\end{equation}
If $z$ has globally minimal projection, then integer points in~$P$ can project only onto $\proj{\vertices(P)}{V}$, which completes the first step.
The argument behind this will be made explicit in Proposition~\ref{prop:fullorbit-infinite-coresets} below.
However, we will see later that for irrational subspaces there is no point with global minimal projection (cf.~Lemmas~\ref{lem:projection-positive}~and~\ref{lem:projection-smaller-eps}).
In this case the following weaker condition suffices.
We say that the point $z$ has \defterm{locally minimal} projection onto~$V$ if
\begin{equation}\label{eq:projection-locally-minimal}
\norm{\proj{z}{V}} \leq \norm{\proj{z'}{V}}\quad\text{for all $z' \in \zzk{k}$ with $\norm{z'} \leq \norm{z}$.}
\end{equation}
Since only points with $\norm{z'} \leq \norm{z}$ can lie in the orbit polytope $P = \conv \Gamma z$, it is enough to control the projection of these points.

For the second step of the outline above, proving lattice-freeness of $P$, our argument is based on the stabilizer group of the vertex $z$ and its projection $\proj{z}{V}$.
We will need the following lemma.
\begin{lemma}\label{lem:projection-stabilizer-supergroup}
 Let $\Gamma\leq\Symmet{n}$ and $V$ an invariant subspace of~$\Gamma$.
 For any $z \in \zz^n$ we have $\stab_\Gamma(z) \leq \stab_\Gamma(\proj{z}{V})$.
\end{lemma}
\begin{proof}
 Let $\gamma \in \stab_\Gamma(z)$, thus $\gamma z = z$.
 This implies $\gamma (\proj{z}{V}+\proj{z}{W}) = \proj{z}{V}+\proj{z}{W}$.
 Hence $\gamma \proj{z}{V} - \proj{z}{V} = \proj{z}{W} - \gamma \proj{z}{W}$.
 The only element in $V\cap W$ is the zero vector.
 Therefore $\gamma \in\stab_\Gamma(\proj{z}{V})$.
\end{proof}

\begin{proposition}\label{prop:fullorbit-infinite-coresets}
 Let $\Gamma\leq\Symmet{n}$ be a permutation group and $V$ an invariant subspace of~$\Gamma$.
 Let $z\in\zz^n$ have locally minimal projection for $V$.
 Then $z$ is a core point for~$\Gamma$ if and only if $z$ is a core point for $\stab_\Gamma(\proj{z}{V})$.
\end{proposition}
\begin{proof}
 Because $\stab_\Gamma(\proj{z}{V})$ is a subgroup of $\Gamma$, we only have to prove the ``if''-part.
 For this let $y$ be an integer point in $\conv\ac{z}{\Gamma}$.
 We can write $y$ as a convex combination 
\begin{equation}\label{eq:fullorbit-convex-combination}
y = \sum_{\gamma \in \Gamma} \lambda_\gamma \ac{z}{\gamma }
\end{equation}
with $0 \leq \lambda_\gamma \leq 1$ and $\sum_{\gamma \in \Gamma} \lambda_\gamma = 1$.
This yields:
 \begin{equation}\label{eq:jensen}
\begin{split}
\|\proj{z}{V}\|^2 \leq \|\proj{y}{V}\|^2 &= \|\proj{\left ( \sum_{\gamma \in \Gamma} \lambda_\gamma \ac{z}{\gamma } \right)}{V}\|^2 \\
&\leq \sum_{\gamma \in \Gamma} \lambda_\gamma \|\proj{\left(\ac{z}{\gamma }\right)}{V}\|^2
=\|\proj{z}{V}\|^2.
\end{split}
\end{equation}
The first inequality holds because we assumed that $z$ has locally minimal projection.
The second inequality holds because of the convexity of a norm square and Jensen's inequality.
The last equation holds since $\norm{\proj{(\gamma z)}{V}} = \norm{\gamma (\proj{z}{V})} = \norm{\proj{z}{V}}$.
For this we use Lemma~\ref{lem:projection-invariant-subspace-commute-group-action} and that the linear representation of~$\gamma$ is an orthogonal matrix.
Note that the left- and right-most terms of~\eqref{eq:jensen} are the same, so we must in fact have equality.

 Since the squared norm is strictly convex on $V$, equality in \eqref{eq:jensen} holds if and only if there is a coset $\gamma_0 \stab_\Gamma(\proj{z}{V})$ such that $\sum_{\gamma \in \gamma_0\stab_\Gamma(\proj{z}{V})} \lambda_\gamma = 1$.
Plugging this into~\eqref{eq:fullorbit-convex-combination} yields
 $$\inv{\gamma_0}y = \sum_{\gamma \in \stab_\Gamma(\proj{z}{V})} \lambda_\gamma \gamma z.$$
 Since $z$ is a core point for $\stab_\Gamma(\proj{z}{V})$, we must have $\inv{\gamma_0}{y} \in \stab_\Gamma(\proj{z}{V})z$.
 Hence, the point $y$ lies also in the orbit $\Gamma z$.
 From this we conclude that $z$ is a core point for $\Gamma$.
\end{proof}

\subsection{Imprimitive groups}\label{sec:infinite-rational-subspace}

We start this section with a specialization of Proposition~\ref{prop:fullorbit-infinite-coresets} 
and then prove that imprimitive groups have an infinite number of core points (up to translation).
\begin{corollary}\label{cor:direction-corepoint}
 Let $\Gamma\leq\Symmet{n}$ be a permutation group and $\rr^n = \fix (\Gamma) \oplus V \oplus W$ a decomposition into $\Gamma$-invariant subspaces.
 Let $z \in \zzk{k}$ be a core point for~$\stab_\Gamma(\proj{z}{V})$ with globally minimal projection.
 Let $w \in W \cap \zz^n$ be such that $\stab_\Gamma(\proj{z}{V}) \leq \stab_\Gamma(w)$.
 Then for all $m\in \zz$ the polytope $P_m \defeq \conv \ac{(z + m w)}{\Gamma}$ contains no integer points except its vertices.
\end{corollary}
\begin{proof}
 To prove that $P_m$ is lattice-free we apply Proposition~\ref{prop:fullorbit-infinite-coresets}.
 Since $z$ has globally minimal projection onto $V$, so does $z + m w$.
 In particular, $z + m w$ thus also has locally minimal projection.
 It remains to show that $z+mw$ is a core point for $\stab_\Gamma(\proj{z}{V})$.
 Because of the inclusion $\stab_\Gamma(\proj{z}{V}) \leq \stab_\Gamma(w)$, we have that
 $$P'_m\defeq\conv \left( \ac{(z+mw)}{\stab_\Gamma(\proj{z}{V})} \right) = mw + \conv  \left( \ac{z}{\stab_\Gamma(\proj{z}{V})} \right).$$
 Because $z$ is a core point for $\stab_\Gamma(\proj{z}{V})$ by assumption of the corollary, 
this shows that the polytope $P'_m$ is lattice-free.
 Hence, $z+mw$ is a core point for $\stab_\Gamma(\proj{z}{V})$ and thus also for $\Gamma$ by Proposition~\ref{prop:fullorbit-infinite-coresets}.
\end{proof}
\begin{example}\label{ex:infinite-core-points-c4}
 As an example we consider the cyclic group $\Cycl{4} = \langle (1\,2\,3\,4)\rangle$.
 The arguments here will be generalized to imprimitive groups later in this section.
 The vector $w \defeq (1,-1,1,-1)^\tp$ spans a one-dimensional invariant subspace.
 Its orthogonal complement, besides the fixed space, is spanned by $v\defeq (1,0,-1,0)^\tp$ and $v'\defeq (0,1,0,-1)^\tp$.

 For applying Corollary~\ref{cor:direction-corepoint}, let $V \defeq \lin \{v,v'\}$ and $W \defeq \lin \{w\}$.
 We will see in Lemma~\ref{lemma:e1-minimal} that $e_1 \defeq (1,0,0,0)^\tp$ is a core point for $\Cycl{4}$ with globally minimal projection on~$V$.
 We compute $\proj{e_1}{V} = \frac{1}{2}v$, hence the stabilizer $\stab_{\Cycl{4}}(\proj{e_1}{V})$ is trivial.
 Therefore we may choose any integer direction from~$W$.
 As these are all multiples of~$w$, this yields the sequence of core points, $e_1 + m w = (1+m, -m, m, -m)^\tp$.

 If we swap the roles of $V$ and $W$ in Corollary~\ref{cor:direction-corepoint}, we still have that $e_1$ has globally minimal projection on~$W$ (again cf.~Lemma~\ref{lemma:e1-minimal}).
 Its projection is $\proj{e_1}{W} = \frac{1}{4} w$ with stabilizer $\stab_{\Cycl{4}}(\proj{e_1}{W}) = \langle (1\,3)(2\,4)\rangle$.
 Since all non-zero elements from $V$ have trivial stabilizer, we cannot find a suitable integer direction $v \in V$ that is compatible with the stabilizer condition of Corollary~\ref{cor:direction-corepoint}.

 However, we may also work with Proposition~\ref{prop:fullorbit-infinite-coresets} directly.
 Let $a v + b v'$ with $a,b\in \zz$ be an arbitrary integer direction in $V$.
 By Proposition~\ref{prop:fullorbit-infinite-coresets}, 
$$p(a,b) \defeq e_1 + a v + b v' =  (1+a,\, b,\, -a,\, -b)^\tp $$ 
is a core point for $\Cycl{4}$ if and only if it is a core point for $\stab_{\Cycl{4}}(\proj{e_1}{W})= \langle (1\,3)(2\,4)\rangle$.
 The orbit polytope $\conv\stab_{\Cycl{4}}(\proj{e_1}{W}) p(a,b)$ has only two vertices,
\begin{align*}
 u &\defeq (1+a,\, b,\, -a,\, -b)^\tp,\quad\text{and}\\
 u' &\defeq (-a,\, -b,\, 1+a,\, b)^\tp.
\end{align*}
 Consider a proper convex combination $\lambda u + (1-\lambda) u'$ on the line segment between $u$ and $u'$ with $0 < \lambda < 1$.
 If $\lambda u + (1-\lambda) u'$ is integral, then 
 looking at the first coordinate shows that $\lambda (1+a) +(1-\lambda) (-a) = (2a+1)\lambda$ must be an integer.
 Looking at the second coordinate, we similarly obtain that $2b\lambda$ must be an integer.
 If $\gcd(2a+1,2b) > 1$ this is possible for $\lambda=1 / \gcd(2a+1,2b)$.
 If $b = 0$, the second condition is automatically fulfilled and the first condition is satisfiable if $a \notin \{-1,0\}$.
 We have therefore proven: $p(a,b)$ is a core point for $\Cycl{4}$ if and only if $\gcd(2a+1,2b) = 1$ (with our convention $\gcd(x,0) = \card{x}$).
 Figure~\ref{fig:orbit-polytope-C4-2} depicts instances of lattice-free orbit tetrahedra of $p(a,b)$ for $(a,b) \in \{(0,0),(0,1),(0,2),(0,3),(0,4)\}$.
 $\blacksquare$
\end{example}
\begin{figure}[htb]
\begin{center}
 \includegraphics[scale=0.16]{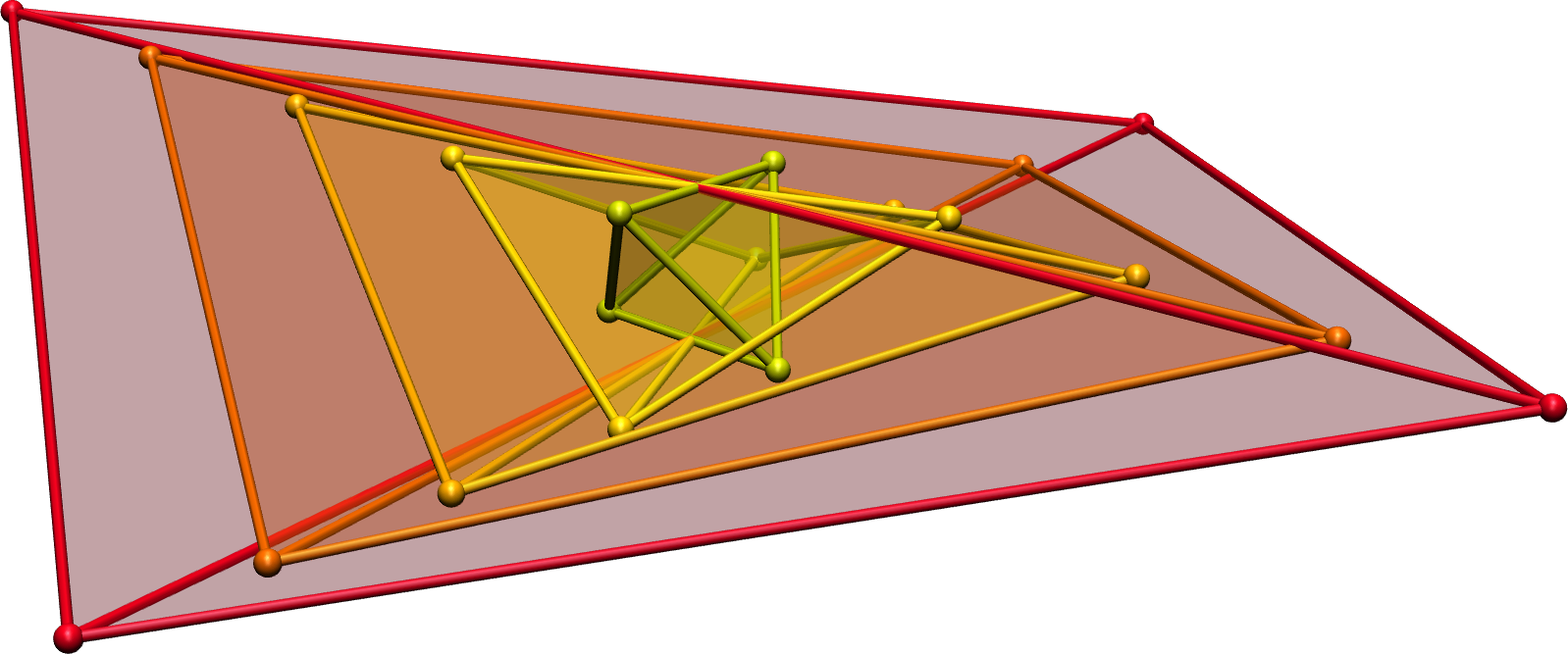}
\end{center}
\caption{Impression of an infinite sequence of lattice-free orbit polytopes for $\Cycl{4}$}\label{fig:orbit-polytope-C4-2}
\end{figure}
We will show that the conditions of Corollary~\ref{cor:direction-corepoint} are satisfied for imprimitive groups.
However, it can also be applied to other groups with rational subspaces if a suitable direction is found, which may however be difficult.

Recall from Section~\ref{sec:permutation-group-basics} the definition of an imprimitive permutation group.
For each imprimitive permutation group $\Gamma$ acting on $\setOneTo{n}$, there is a partition of $\setOneTo{n} = \bigsqcup_{i=1}^{B} \Omega_i$ such that $\Gamma$ acts on the $B$ sets $\Omega_i$.
For every $\gamma \in \Gamma$ and index $i\in\setOneTo{B}$ there exists an index $j$ such that $\gamma \Omega_i = \Omega_j$.
Every such block $\Omega_i$ has size $S = \frac{n}{B}$.
These blocks induce a rational invariant subspace of $\Gamma$ in the following way.
Let
\begin{equation}\label{eq:imprimitive-invariant-subspace-basis}
 \cin{u}{j} \defeq \sum_{i \in \Omega_j} e_i \in \zz^n
\end{equation}
be the characteristic vector of $\Omega_j$.
Then the vectors $\cin{u}{1}, \dots, \cin{u}{B}$ form an orthogonal basis of an $\Gamma$-invariant subspace of $\rr^n$.
We call this $B$-dimensional subspace $U_\Omega\defeq \lin \{\cin{u}{1}, \dots, \cin{u}{B}\}$.
Since $\vones = \sum_{j=1}^{B} \cin{u}{j}$, we know that $U$ contains $\fix(\Gamma) = \lin \vones$.
We can thus split $U$ into a direct sum $U_\Omega = \lin \vones \oplus W_\Omega$ for another rational invariant subspace $W_\Omega$.
Furthermore, there is an invariant subspace $V_\Omega$ which is the orthogonal complement of $U_\Omega$ in $\rr^n$.
In total we obtain for each block system~$\Omega$ the following decomposition into rational invariant subspaces:
\begin{equation}\label{eq:imprimitive-invariant-subspaces}
 \rr^n = \underbrace{\lin\vones\, \oplus\, W_\Omega}_{U_\Omega}\, \oplus\, V_\Omega.
\end{equation}

\begin{example}\label{ex:invariant-subspace-C6}
 As an example we consider the cyclic group $\Cycl{6} = \langle (1\,2\,3\,4\,5\,6)\rangle$.
 The group action of $\Cycl{6}$ is imprimitive as it preserves the partition $\Omega=\left\{\{1,3,5\},\{2,4,6\}\right\}$.
 The corresponding invariant subspace $U_\Omega$ is $\lin \{ (1,0,1,0,1,0)^\tp, (0,1,0,1,0,1)^\tp \}$.
 For its non-fixed summand we obtain $W_\Omega = \lin (1,-1,1,-1,1,-1)^\tp$.

 Note that the block system and the corresponding decomposition~\eqref{eq:imprimitive-invariant-subspaces} is not unique.
 For instance, the group $\Cycl{6}$ has another block system $\Omega' = \left\{\{1,4\},\{2,5\},\{3,6\}\right\}$.
 This corresponds to $U_{\Omega'} = \lin \{ (1,0,0,1,0,0)^\tp, (0,1,0,0,1,0)^\tp, (0,0,1,0,0,1)^\tp \}$ and $W_{\Omega'} = \lin \{ (2,-1,-1,2,-1,-1)^\tp, (-1,2,-1,-1,2,-1)^\tp \}$.
 $\blacksquare$
\end{example}

With these invariant subspaces $V_\Omega$ and $W_\Omega$ we show that the conditions of Corollary~\ref{cor:direction-corepoint} are fulfilled for imprimitive groups.

\begin{theorem}\label{thm:imprimitive-infinite-coreset}
 Let $\Gamma \leq \Symmet{n}$ act imprimitively, i.\,e. the permutation action of $\Gamma$ preserves a block system with blocks of size $1 < S < n$.
 If $k$ is not a multiple of $S$, then $\Gamma$ has infinitely many core points in layer $\zzk{k}$.
\end{theorem}
\begin{proof}
The proof of this theorem follows immediately from applying Corollary~\ref{cor:direction-corepoint} to the following Lemma~\ref{lemma:e1-minimal}.
By the latter, we find a core point $\cin{z}{k}$ in the claimed layers with globally minimal projection onto $V_\Omega$.
Moreover, it produces a non-zero direction $w \in W_\Omega \cap \zz^n$ such that $\stab_\Gamma(\proj{z^{(k)}}{V_\Omega}) \leq \stab_\Gamma(w)$.
Therefore, for every $m\in \zz$, the point $\cin{z}{k} + mw$ is a core point by Corollary~\ref{cor:direction-corepoint}.
Since $w$ is not the zero vector, these core points are different for varying $m$.
\end{proof}

\begin{lemma}\label{lemma:e1-minimal}
 Let $\Gamma \leq \Symmet{n}$ act imprimitively, i.\,e. the permutation action of $\Gamma$ preserves a block system with blocks of size $1 < S < n$.
 For the corresponding invariant subspaces $U_\Omega, V_\Omega, W_\Omega$ from \eqref{eq:imprimitive-invariant-subspaces} the following holds:
 If $k$ is not a multiple of $S$, then
 there exist a core point $z^{(k)} \in \zzk{k}$ with globally minimal projection onto $V_\Omega$ and a non-zero direction $w \in W_\Omega \cap \zz^n$ such that $\stab_\Gamma(\proj{z^{(k)}}{V_\Omega}) \leq \stab_\Gamma(w)$.
\end{lemma}
\begin{proof}
 To keep notation as simple as possible, we assume w.l.o.g.\ that the first of the blocks $\Omega_1,\dots,\Omega_B$ of $\Gamma$ is $\Omega_1 = \{1,\dots,S\}$.
 Consider an arbitrary $z \in \zz^n$.
 We compute the squared norm of the projection onto~$V_\Omega$ as 
\begin{equation}\label{eq:imprimitive-projection-V-minimal}
\begin{split}
 \|\proj{z}{V_\Omega}\|^2 &= \|z\|^2 - \|\proj{z}{U_\Omega}\|^2 \\&= \left ( \sum_{b = 1}^{B} \sum_{j\in\Omega_b} z_j^2 \right ) - \left ( \frac{1}{S} \sum_{b=1}^{B} \left(\sum_{j\in\Omega_b} z_j\right)^2 \right ) \\
 &= \frac{1}{S} \sum_{b = 1}^{B} \ \sum_{\substack{i,j\in\Omega_b\\i < j}} \left(z_i-z_j\right)^2.
\end{split}
\end{equation}
Looking at this sum of squares, we observe that the total expression is minimized if inside each block $\Omega_b$ the coordinates differ in the least possible way and the total number of blocks with non-zero contribution is minimized.
Let $l\in\{0,\ldots, S-1\}$ be congruent to $k \bmod S$.
Then the point
\begin{equation}\label{eq:imprimitive-layer-minimum}
 \cin{z}{k} = \sum_{i=1}^{l} e_i + \sum_{j=2}^{\floor{\frac{k}{S}}+1} \cin{u}{j}.
\end{equation}
with $\cin{u}{j}$ as in~\eqref{eq:imprimitive-invariant-subspace-basis}
satisfies this condition and hence has globally minimal projection.
As a sum of squares, the projection in~\eqref{eq:imprimitive-projection-V-minimal} can be zero if and only if $k$ is a multiple of $S$.
Thus, $\cin{z}{k}$ has non-zero length if $k$ is not a multiple of $S$.
The choice for the minimum in~\eqref{eq:imprimitive-layer-minimum} is not the most obvious, but it has the advantage that it is a universal core point because it has coordinates with only zeros and ones.

Now that we have found a core point $\cin{z}{k}$ with globally minimal projection, it remains to find a suitable non-zero direction $w \in W_\Omega\cap\zz^n$.
For this we need the stabilizer $\stab_\Gamma(\proj{\cin{z}{k}}{V_\Omega})$ to be contained in $\stab_\Gamma(w)$.
To compute the projection~$\proj{\cin{z}{k}}{V_\Omega}$ we again use our explicit basis for~$U_\Omega$.
Looking again at~\eqref{eq:imprimitive-layer-minimum} we see that $\proj{\cin{z}{k}}{V_\Omega} = \proj{\cin{z}{k+S}}{V_\Omega}$ since the vectors differ only in summands from $U_\Omega$, which is the orthogonal complement of $V_\Omega$.
For the projection we may thus assume w.l.o.g.\ that $k < S$ and we compute
\begin{equation}\label{eq:zk-projection}
\proj{\cin{z}{k}}{V_\Omega} = \cin{z}{k} - \proj{\cin{z}{k}}{U_\Omega} = \left(\sum_{i=1}^k e_i \right)- \frac{k}{S} \cin{u}{1} = \sum_{i=1}^k \left(1 - \frac{k}{S}\right) e_i - \sum_{i=k+1}^S \frac{k}{S} e_i.
\end{equation}
For the direction $w$ we look at the projection of $\cin{u}{1}$ onto $W_\Omega$, which is $\proj{\cin{u}{1}}{W_\Omega} = \cin{u}{1} - \frac{S}{n}\vones$.
After scaling this gives a non-zero integer vector~$w$ with stabilizer $\stab_\Gamma(w) = \stab_\Gamma(\Omega_1)$.
Looking again at~\eqref{eq:zk-projection}, we observe that $\proj{\cin{z}{k}}{V_\Omega}$ has a zero at coordinate~$i$ if and only if~$i$ is not in~$\Omega_1$.
Thus, the stabilizer of $\proj{\cin{z}{k}}{V_\Omega}$ must be a subgroup of $\stab_\Gamma(\setOneTo{n} \setminus \Omega_1) = \stab_\Gamma(\Omega_1) = \stab_\Gamma(w)$.
\end{proof}

\begin{remark}\label{rem:imprimitive-layer-minima-not-unique}
Note that many points minimize~\eqref{eq:imprimitive-projection-V-minimal}.
As long as they are core points, they are valid alternative choices for $\cin{z}{k}$ in~\eqref{eq:imprimitive-layer-minimum}.
If used in the proof of Theorem~\ref{thm:imprimitive-infinite-coreset}, they may also lead to infinite sequence of core points.
\end{remark}

\begin{example}\label{ex:invariant-subspace-C6-continued}
We continue Example~\ref{ex:invariant-subspace-C6} and construct core points for the cyclic group~$\Cycl{6}$.
We begin with the block system $\Omega=\left\{\{1,3,5\},\{2,4,6\}\right\}$.
We thus have $B = 2$ and size $S = 3$.
Hence, we can expect infinitely many core points in the layers with indices $k= 1, 2, 4, 5$ because these are not multiples of $S$.
The layer minima $\cin{z}{k}$ from Lemma~\ref{lemma:e1-minimal} are given by
\begin{equation*}
 \begin{split}
  \cin{z}{1} &= (1,0,0,0,0,0)^\tp,\\
  \cin{z}{2} &= (1,0,1,0,0,0)^\tp,\\
  \cin{z}{4} &= (1,1,0,1,0,1)^\tp,\\
  \cin{z}{5} &= (1,1,1,1,0,1)^\tp.
 \end{split}
\end{equation*}
The corresponding direction is $w = (1,-1,1,-1,1,-1)^\tp$ from Example~\ref{ex:invariant-subspace-C6}.
Corollary~\ref{cor:direction-corepoint} implies that for every $m \in \zz$ the simplex $\conv \ac{(\cin{z}{k} + m w)}{\Cycl{6}}$ is lattice-free.
In the case $k = 1$, for every $m\in\zz$ the simplex given by the orbit of 
\begin{equation}\label{eq:cyclic6-example-corepoint}
\cin{z}{1} + m w = (1+m,-m,m,-m,m,-m)^\tp \in \zzk{1} 
\end{equation}
is lattice-free.

Note that for the layer with index $k=3$ this construction did not produce an infinite sequence of simplices.
But we can find such a sequence by looking at the other invariant block system of $\Cycl{6}$, which is $\Omega'=\left\{\{1,4\},\{2,5\},\{3,6\}\right\}$ with size $S = 2$.
Using this, we find infinitely many core points in the layers $1$, $3$ and $5$ by Theorem~\ref{thm:imprimitive-infinite-coreset}.
The corresponding layer minima are
\begin{equation*}
 \begin{split}
  \cin{z}{1} &= (1,0,0,0,0,0)^\tp,\\
  \cin{z}{3} &= (1,1,0,0,1,0)^\tp,\\
  \cin{z}{5} &= (1,1,1,0,1,1)^\tp.
 \end{split}
\end{equation*}
As direction $w$ we choose a multiple of $\proj{\cin{u'}{1}}{W_{\Omega'}} = \frac{1}{3}(2,-1,-1,2,-1,-1)^\tp$ such that the vector is integral.
In the case $k = 3$, for instance, the simplex given by the orbit of 
\begin{equation}\label{eq:corepoint-series-cyclic6-layer3}
 z^{(3)} + m \cin{w}{1} = (1+2m,\,1-m,\,-m,\,2m,\,1-m,\,-m)^\tp \in \zzk{3}
\end{equation}
is lattice-free for every $m\in\zz$.
An alternative choice for $\cin{z}{3}$ could be $(1,0,0,1,0,1)$ (cf.~Remark~\ref{rem:imprimitive-layer-minima-not-unique}), leading to the sequence of core points 
\begin{equation}\label{eq:corepoint-series-cyclic6-layer3-alt}
(1+2m,\,-m,\,-m,\,1+2m,\,-m,\,1-m)^\tp \in \zzk{3} 
\end{equation}
for $m\in\zz$. 
The core points described by~\eqref{eq:corepoint-series-cyclic6-layer3-alt} and~\eqref{eq:corepoint-series-cyclic6-layer3} are different.
To see this we observe that in \eqref{eq:corepoint-series-cyclic6-layer3-alt} two consecutive coordinates have the same value $-m$; this does happen in \eqref{eq:corepoint-series-cyclic6-layer3}.
Besides these constructions, there are entirely different ones that yield infinite sequences for~$\Cycl{6}$.

For instance, one can check that for every $a,b \in \zz$ the simplex given by the orbit of $(1,a,b,0,-a,-b)^\tp \in \zzk{1}$ is lattice-free.
We have already seen a proof of similar constructions for the special case~$\Cycl{4}$ in Example~\ref{ex:infinite-core-points-c4}.
More examples can be found in \cite{KatrinPhD, ThomasPhD}.
$\blacksquare$
\end{example}

\subsection{Irrational subspaces}\label{sec:infinite-irrational-subspace}

In this section we will construct core points using irrational invariant subspaces.
The main result will be the following.
\begin{theorem}\label{thm:irrational-infinite-coreset}
 Let $\Gamma \leq \Symmet{n}$ have an irrational invariant subspace.
 If $k$ is not a multiple of $n$, then $\Gamma$ has infinitely many core points in layer $\zzk{k}$.
\end{theorem}

To prove this theorem, we begin this section with an adaption of Proposition~\ref{prop:fullorbit-infinite-coresets}.
\begin{corollary}\label{cor:fullorbit-corepoint}
 Let $\Gamma\leq\Symmet{n}$ and let $V$ be an invariant subspace of~$\Gamma$.
 Let $z \in \zzk{k}$ be an integer point with locally minimal projection.
 Moreover, let $\stab_\Gamma(z) = \stab_\Gamma(\proj{z}{V})$.
 Then $z$ is a core point for $\Gamma$.
\end{corollary}
\begin{proof}
 The minimality condition is the same as in Proposition~\ref{prop:fullorbit-infinite-coresets}.
 If $\stab_\Gamma(z) = \stab_\Gamma(\proj{z}{V})$, then the orbit of $\ac{z}{\stab_\Gamma(\proj{z}{V})}$ consists only of a single element,
 showing that $z$ is a core point for $\stab_\Gamma(\proj{z}{V})$.
 Thus by Proposition~\ref{prop:fullorbit-infinite-coresets} $z$ is a core point for $\Gamma$.
\end{proof}

In order to apply Corollary~\ref{cor:fullorbit-corepoint} for 
the proof of Theorem~\ref{thm:irrational-infinite-coreset}, 
we show that its prerequisites are satisfied for an irrational invariant subspace. 
First we show in Lemma~\ref{lem:QI-fullorbit} that the stabilizer condition holds.
Lemma~\ref{lem:projection-smaller-eps} and Lemma~\ref{lem:projection-positive} together show that the local minimality condition is fulfilled.

\begin{lemma}\label{lem:QI-fullorbit}
 Let $\Gamma \leq \Symmet{n}$ and let $V$ be an irrational invariant subspace of~$\Gamma$.
 Then $\stab_\Gamma(z) = \stab_\Gamma(\proj{z}{V})$ for any $z \in \zz^n$.
\end{lemma}
\begin{proof}
 We have already proven $\stab_\Gamma(z) \leq \stab_\Gamma(\proj{z}{V})$ in Lemma~\ref{lem:projection-stabilizer-supergroup}.
 For the reverse direction let $\rr^n = \lin \vones \oplus V \oplus W$. 
 Then $W$ must be an irrational invariant subspace because~$V$ is irrational.
 We consider a $\gamma \in \Gamma \setminus \stab_\Gamma(z)$ and show $\gamma \not \in \stab_\Gamma(\proj{z}{V})$.
 For $z=0$ the statement is obviously true, so let $z\not = 0$.
 Then $$\gamma z - z \; = \; \proj{( \gamma z - z )}{V} + \proj{( \gamma z - z )}{W}$$ is a non-zero integral vector.
 As $V$ and $W$ are irrational subspaces, both projections on the right must be non-zero,
 showing in particular $\proj{( \gamma z - z )}{V} = \proj{\gamma z }{V}- \proj{z}{V} \not= 0$.
 Hence $\gamma \not \in \stab_\Gamma(\proj{z}{V})$.
\end{proof}

\begin{lemma}\label{lem:projection-positive}
 Let $\Gamma \leq \Symmet{n}$ and let $V$ be an irrational invariant subspace of~$\Gamma$.
 Then for all $k \in \setOneTo{n-1}$ and every $z\in \zzk{k}$ it holds that $\|\proj{z}{V}\| > 0$.
\end{lemma}
\begin{proof}
 Let $\rr^n = \lin \vones \oplus V \oplus W$. Then $W$ is an irrational invariant subspace of~$\Gamma$.  
 We know that $\proj{z}{V}$ is the zero vector if and only if $z \in \lin \vones \oplus W$.
 This is in turn equivalent to the rational vector $z - \frac{k}{n}\vones$ lying in~$W$.
 Because~$W$ is irrational, the only rational vector it contains is the zero vector.
 Thus, the projection $\proj{z}{V}$ can be zero only if~$k$ is an integral multiple of~$n$.
\end{proof}

\begin{lemma}\label{lem:projection-smaller-eps}
 Let $\Gamma \leq \Symmet{n}$ and let $V$ be an irrational invariant subspace of~$\Gamma$.
 Then for every $\eps > 0$ and $k \in \setOneTo{n-1}$ there exists a vector $z \in \zzk{k}$ such that $\|\proj{z}{V}\| < \eps$.
\end{lemma}

For the proof of Lemma~\ref{lem:projection-smaller-eps}, we use two auxiliary statements.
We begin with the symmetry of the projection matrix $P_V = (\proj{e_i}{V})_{i \in \setOneTo{n}} \in \rr^{n\times n}$, which maps $\rr^n$ onto an invariant subspace~$V$.
\begin{lemma}\label{lem:proj-matrix-symmetric}
For the orthogonal projection to a linear subspace $V$ holds:
\begin{enumerate}
 \item $\scp{\proj{e_i}{V}}{e_j} = \scp{\proj{e_i}{V}}{\proj{e_j}{V}}$
 \item The projection matrix $P_V = (\proj{e_i}{V})_{i \in \setOneTo{n}} \in \rr^{n\times n}$ is symmetric.
\end{enumerate}
\end{lemma}
\begin{proof}
 Let $v_1,\dots,v_d$ be an orthonormal basis for $V$.
 \begin{equation*}
 \begin{split}
  \scp{\proj{e_i}{V}}{\proj{e_j}{V}} &= \scp{\sum_{k=1}^d \scp{e_i}{v_k} v_k}{\sum_{l=1}^d \scp{e_j}{v_l} v_l} \\
  &= \sum_{k=1}^d \scp{e_i}{v_k} \scp{e_j}{v_k} = \scp{\proj{e_i}{V}}{e_j} 
 \end{split}
 \end{equation*}
The symmetry in the second part follows from the symmetry of the scalar product in $\scp{\proj{e_i}{V}}{e_j} = \scp{\proj{e_i}{V}}{\proj{e_j}{V}} = \scp{\proj{e_j}{V}}{\proj{e_i}{V}} = \scp{\proj{e_j}{V}}{e_i}$.
\end{proof}

The main ingredient to prove Lemma~\ref{lem:projection-smaller-eps} is Kronecker's Theorem, which is reproduced below as given in \cite[p.~80]{Schrijver1998}.
\begin{theorem}[Kronecker's Theorem]\label{thm:kronecker}
 Let $A \in \rr^{m\times n}$ and let $b\in \rr^n$.
 Then the following two statements are equivalent:
 \begin{enumerate}
  \item for each $\varepsilon > 0$ there is an $x\in \zz^n$ with $\|Ax-b\| < \varepsilon$;
  \item for each $y\in \rr^m$ the implication $A^\tp y \in \zz^n \; \Rightarrow \; b^\tp y \in \zz$ is true.
 \end{enumerate}
\end{theorem}

\begin{proof}[Proof of Lemma~\ref{lem:projection-smaller-eps}]
 Using the projection matrix $P_V = (\proj{e_{i}}{V})_{i \in \setOneTo{n}} \in \rr^{n\times n}$, our goal is to show that for every $\varepsilon > 0$ there exists a $z \in \zzk{k}$ with $\|P_V z\| < \varepsilon$.
 Let $B \in \rr^{n\times(n-1)}$ be the matrix whose columns consist of the vectors $\cin{b}{i} \defeq e_{i+1} - e_{i}$ for $i \in \setOneTo{n-1}$.
 We can write every $z \in \zzk{k}$ as $z = k e_{1} + B z'$ for a suitable $z' \in \zz^{n-1}$.
 Thus, we have to show that for every $\varepsilon > 0$ we find a $z' \in \zz^{n-1}$ such that
\begin{equation}\label{eq:inhomogeneous-diophantine-approximation}
 \|k P_V e_{1} + P_V B z'\| < \varepsilon.
\end{equation}
 Kronecker's Theorem states that this is equivalent to an implication concerning the integrality of $(P_V B)^\tp y$ and $(P_V e_{1})^\tp y$ for $y \in \rr^n$.
 Using the symmetry of $P_V$ from Lemma~\ref{lem:proj-matrix-symmetric}, we have to show that $B^\tp y' \in \zz^n$ implies $(e_{1})^\tp y' \in \zz$ where $y' \defeq P_V y = \proj{y}{V}$ is the projection of $y$ onto $V$.
 
 Let us assume that $B^\tp y' \in \zz^n$ holds.
 We will show that this can only be the case for $y' = 0$, from which we immediately obtain that the implication required by Kronecker's Theorem is satisfied.
 From $B^\tp y' \in \zz^n$ we infer that for all $\cin{b}{i}$ we must have $\scp{\cin{b}{i}}{y'} \in \zz$.
 Thus, we can write $y'$ as $y' = \zeta \vones + u$ for some $\zeta \in \rr$ and an integral vector $u \in \zz^n$.
 Since $y'$ lies in $V$, we know that $0 = \scp{\vones}{y'} = n \zeta + \scp{\vones}{u}$.
 This shows that $\zeta$ must be rational number.
 Hence, $y'$ must be a rational vector.
 The only rational vector lying in the irrational invariant subspace $V$ is the zero vector.
\end{proof}

Now we have all the ingredients for the proof of our main result of this section:

\begin{proof}[Proof of Theorem~\ref{thm:irrational-infinite-coreset}]
  Lemma~\ref{lem:projection-positive} together with Lemma~\ref{lem:projection-smaller-eps} show that for every $k \in \setOneTo{n-1}$ and every $\eps > 0$ we find an integer point $z \in \zzk{k}$ such that $0 < \|\proj{z}{V}\| < \eps$ and, by choosing one with minimal norm, $\|\proj{z}{V}\| \leq \|\proj{z'}{V}\|$ for all $z' \in \zzk{k}$ with $\|z'\| \leq \|z\|$.

  By letting $\eps$ approach zero, we thus obtain a sequence of mutually distinct points $\cin{z}{1}, \cin{z}{2}, \dots \in \zzk{k}$,
  which by construction each satisfy the minimality condition of Corollary~\ref{cor:fullorbit-corepoint}.
  Lemma~\ref{lem:QI-fullorbit} shows that also the stabilizer condition of Corollary~\ref{cor:fullorbit-corepoint} is automatically fulfilled.
  Hence, each of these points $\cin{z}{1}, \cin{z}{2}, \dots \in \zzk{k}$ is a core point.
\end{proof}

We close the section with an example of an infinite sequence of core points that can be derived from Theorem~\ref{thm:irrational-infinite-coreset}. 
For a detailed discussion we refer to \cite[Section~5.2.2]{ThomasPhD}.

\begin{example}
 Let $\Cycl{5} = \langle (1\,2\,3\,4\,5)\rangle$ be the cyclic group of order five.
 Moreover, let $f_j$ be the $j$-th Fibonacci number.
 For every $j$ the point $\cin{z}{j} = (0,f_j,f_j,0,f_{j+1})^\tp$ is a core point for $\Cycl{5}$.
\end{example}

\section*{Acknowledgements}

We thank Reinhard Knörr and Michael Joswig for valuable discussions
and we thank Erik Friese and Frieder Ladisch for helpful comments on a previous version of the paper.
We also thank the anonymous referees for their beneficial remarks.


\providecommand{\bysame}{\leavevmode\hbox to3em{\hrulefill}\thinspace}
\providecommand{\MR}{\relax\ifhmode\unskip\space\fi MR }
\providecommand{\MRhref}[2]{%
  \href{http://www.ams.org/mathscinet-getitem?mr=#1}{#2}
}
\providecommand{\href}[2]{#2}


\bibliographystyle{alpha}

\section*{References}
\subsection*{Bibliography}

\renewcommand{\section}[2]{}%




\subsection*{Software}





\end{document}